\documentclass{amsart}
\usepackage{psfrag,graphicx,epstopdf,pinlabel,hyperref,fancyhdr}

\newtheorem{theorem}{Theorem}[section]
\newtheorem{lemma}[theorem]{Lemma}
\newtheorem{corollary}[theorem]{Corollary}

\newtheorem{proposition}[theorem]{Proposition}

\newcommand{\ZZ}{\mathbb{Z}} 
\newcommand{\NN}{\mathbb{N}}
\newcommand{\QQ}{\mathbb{Q}}
 
\newcommand{\RR}{\mathbb{R}}

\newcommand{\DD}{\mathcal{D}}
\newcommand{\GG}{\Gamma}
\newcommand{\RP}{\mathbb{R}\mathbb{P}}

\def\vol{\mathrm{Vol}}

\def\whop{\widehat{\Omega}}

\def\O{\mathrm{O}}

\def\mod{\mbox{mod}\ }

\def\HH{\mathbb{H}}

\def\MC{\mathcal{C}}
\def\MN{\mathcal{N}}

\def\PP{\mathcal{P}}

\def\GG{\mathcal{G}}
\def\TT{\mathcal{T}} 
\def\aa{\alpha} 
\def\QQ{\mathcal{Q}}
\def\RRR{\mathcal{R}}
\def\MRR{\mathcal{R}}

\def\MC{\mathcal{C}}
\def\MN{\mathcal{N}}
\def\ME{\mathcal{E}}

\def\PP{\mathcal{P}}

\def\GG{\mathcal{G}}
\def\TT{\mathcal{T}} 
\def\aa{\alpha} 
\def\bb{\beta} 
\def\cc{\gamma} 
\def\QQ{\mathcal{Q}}

\def\RRR{\mathcal{R}}
\def\S{\mathcal{S}}
\def\wt{\widetilde}

\def\wh{\widehat}
\def\split{\setminus \!\! \setminus}
\def\supp{\mbox{Supp}}
\def\MC{\mathcal{C}}
\def\MN{\mathcal{N}}

\newcommand{\MD}{\mathcal{D}}
\def\K{\mathcal{K}}
\begin{document}

\title[Volume estimates for hyperbolic polyhedra]{Two--sided
combinatorial volume bounds for non--obtuse hyperbolic polyhedra}
\author{Christopher K. Atkinson}

\address{Department of Mathematics\\Temple University}
\email{ckatkin@temple.edu} 
\urladdr{http://www.math.temple.edu/~ckatkin}

\begin{abstract}
We give a method for computing upper and lower bounds for the volume of a
non--obtuse hyperbolic polyhedron in terms of the combinatorics of the
$1$--skeleton.  We introduce an algorithm that detects the geometric
decomposition of good $3$--orbifolds with planar singular locus and
underlying manifold $S^3$.  The volume bounds follow from techniques related
to the proof of Thurston's Orbifold Theorem, Schl\"{a}fli's formula, and
previous results of the author giving volume bounds for right--angled
hyperbolic polyhedra.
\end{abstract}

\maketitle

\section{Introduction}

Andreev's theorem gives a complete characterization of non--obtuse hyperbolic
polyhedra of finite volume  in terms of the combinatorics of their
$1$--skeleta labeled by dihedral angles \cite{andreev1,andreev2}.  Andreev's
theorem also states that there is at most one hyperbolic polyhedron having
prescribed $1$--skeleton with a given labeling by dihedral angles, up to
isometry.  Hence the volume of a non--obtuse hyperbolic polyhedron is
determined completely by its $1$--skeleton labeled by dihedral angles.
Computing the exact volume of such a polyhedron in terms of its combinatorics
is a difficult problem.  Milnor \cite{thurstonnotes}, Vinberg \cite{vinberg},
Cho and Kim \cite{chokim}, Murakami and Yano \cite{murakamiyano}, Derevnin
and Mednykh \cite{derevninmednykh}, and Ushijima \cite{ushijima} have given
formulas that compute the volume of various families of hyperbolic tetrahedra
in terms of their dihedral angles.  Kellerhals \cite{kellerhals} gave
formulas that compute the volume of certain cubes and truncated tetrahedra.
More in the spirit of this paper, Sleator, Tarjan, and W.  Thurston showed
that a certain infinite family of obtuse ideal hyperbolic polyhedra obtained by
subdividing the faces of an icosahedron into triangles has volume equal to
$2N\cdot~V_3~-~O(\log(N))$, where $N$ is the number of vertices and
$V_3\approx 1.01494$ is the volume of the regular ideal hyperbolic
tetrahedron \cite{stt}. 

The main result of this paper
is a technique that gives a two--sided combinatorial volume bound for all
hyperbolic polyhedra with non--obtuse dihedral angles in terms of their
$1$--skeleta.  A weak form of our main result is the following
theorem:

\begin{theorem}  \label{superweak}
Let $\PP$ be a non--obtuse hyperbolic polyhedron containing
no prismatic $4$--circuits, $N_4$ degree $4$ vertices, and $N_3$ degree $3$
vertices. Then
	$$\frac{4N_4 + N_3 -8}{32} \cdot V_8  < 
	\vol(\PP) < \frac{2N_4 + 3 N_3 -2}{4} \cdot V_8 + \frac{15N_3 + 20 N_4}{16}
	\cdot V_3 .$$
\end{theorem}

The constant $V_8$ is the volume of the right--angled ideal hyperbolic
octahedron and is approximately $3.66386.$  The constant $V_3$ is the volume
of the regular ideal hyperbolic tetrahedron and is approximately $1.01494.$
A prismatic $k$--circuit is a simple closed curve in the dual graph of the
$1$-skeleton of $\PP$ composed of $k$ distinct edges such that no two of the
edges are contained in a common face. The upper bound holds for all
non--obtuse hyperbolic polyhedra.  It should be noted that it follows from
Andreev's theorem that all finite volume hyperbolic polyhedra with
non--obtuse dihedral angles have only degree $3$ and degree $4$ vertices.   The main result in this paper is a
technique that gives a lower volume bound in the general case where
prismatic $4$ circuits are allowed.

The following corollary follows from Theorem~\ref{superweak} by making the
compromises necessary to write the bounds in terms of the total number of
vertices.

\begin{corollary}
Let $\PP$ be a non--obtuse hyperbolic polyhedron containing
no prismatic $4$--circuits and $N$ vertices. Then
$$\frac{N-8}{32} \cdot V_8 < \vol(\PP) < 4.0166 N - 1.8319.$$
\end{corollary}

We also characterize the smallest--volume Coxeter $n$--prism for each $n \geq
4$.  An \textbf{$n$--prism} is a polyhedron having $1$--skeleton that is the
combinatorial type of an $n$--gon crossed with an interval.  In some sense,
this is the extreme opposite case to Theorem~\ref{superweak} in that
$n$--prisms contain ``many" prismatic $4$-circuits.

\begin{theorem}\label{prismsum}
Suppose that $\PP$ is a non--obtuse hyperbolic $n$--prism with no dihedral
angles in the interval $(\pi/3,\pi/2)$. Then 

$$(n-3)\cdot \vol(C_1(\pi/3)) < \vol(\PP) < \frac{3n-4}{2}\cdot V_8.$$
\end{theorem}

The proof for this theorem is given in
Section~\ref{C:prisms}.  The constant $\vol(C_1(\pi/3))$ is the volume of the
Lambert cube with essential angles equal to $\pi/3$.  Its value is
approximately $.324423.$  See Section~\ref{S:cubes} for more details.  The
restrictions on the dihedral angles in this theorem are necessary.  If the
angles are not bounded away from $\pi/2$, there exist examples of hyperbolic
$n$--prisms with arbitrarily small volume.

The main technique used in this paper is to use Schl\"{a}fli's formula to
control how the volume of a hyperbolic polyhedron changes as its dihedral
angles are varied.  Schl\"{a}fli's formula implies that the volume of a
hyperbolic polyhedron varies inversely with changes in dihedral angles.  
Both the lower and upper bounds are applications of results of the author
from \cite{volpoly} that gives two--sided combinatorial volume bounds for
right--angled hyperbolic polyhedra.  

For the lower bound, the main idea is to
attempt to increase the dihedral angles of a given hyperbolic Coxeter
polyhedron until they are all $\pi/2$.  For a generic hyperbolic
polyhedron, such a deformation is not possible.  To get around this, the
spherical suborbifold decomposition of Petronio and the Euclidean suborbifold
decomposition of Bonahon--Siebenmann will be used to decompose the polyhedron
into components that either do admit a deformation to a right--angled
hyperbolic polyhedron or that correspond to orbifold Seifert--fiber spaces
that can be obtained as a reflection orbifold
\cite{petronio,bonahonsiebenmann}.  We describe an algorithm that produces 
the suborbifolds provided by the decomposition theorems of Petronio and
Bonahon-Siebenmann.  For the components that admit a deformation to a
right--angled hyperbolic polyhedron, we apply theorems from \cite{volpoly}.
For the orbifold Seifert--fiber space case, we classify such polyhedra
completely give a lower bound for their
volume.  

The upper bound is an application of the upper bounds in \cite{volpoly}.  We
exhibit an angle--nonincreasing deformation from any non--obtuse hyperbolic
polyhedron to one with all right angles.  The resulting polyhedron is
obtained from the original by truncating all finite vertices that are
adjacent to at least one other finite vertex.  

The paper is organized as follows:  In Section~\ref{polyeder}, we state
Andreev's theorem and a generalization.  We describe our methods for
decomposing polyhedra in Sections~\ref{algbs} and \ref{C:decomp}.  In
Section~\ref{deformation} we prove the lower bound in Theorem~\ref{superweak}
by applying the decompositions from Section~\ref{C:decomp}.
In Section~\ref{C:prisms}, Theorem~\ref{prismsum} is proved and the
techniques to strengthen Theorem~\ref{superweak} are introduced.
Section~\ref{upperbound} proves the upper bounds in Theorems~\ref{superweak}
and \ref{prismsum} via a stronger theorem.  In the concluding
Section~\ref{summary} the techniques for computing our bounds on any
non--obtuse hyperbolic polyhedron are summarized and an example is given.

\section{Polyhedra and Andreev's theorem}\label{polyeder}

In this section, we introduce the relevant terminology pertaining to polyhedra.  We
also state Andreev's theorem and a generalization.  These theorems classify
non--obtuse hyperbolic polyhedra in terms of the combinatorics of their
$1$--skeleta.

An \textit{abstract polyhedron} is a cell complex on $S^2$ that can be
realized by a convex Euclidean polyhedron.  A theorem of Steinitz says that
realizability as a convex Euclidean polyhedron is equivalent to the
$1$--skeleton of the cell complex being $3$--connected \cite{steinitz}.  A
graph is $3$--connected if the removal of any $2$ vertices along with their
incident open edges leaves the complement connected.  Define a \textit{labeling}
of an abstract polyhedron $P$ to be a function $$\Theta: \text{Edges}(P) \to
(0,\pi).$$  A \textit{non--obtuse} labeling is one where $\text{Image}(\Theta)
\subset (0, \pi/2]$.   A pair $(P, \Theta)$ where $P$ is an abstract
polyhedron and $\Theta$ is a labeling of $P$ is a \textit{labeled abstract
polyhedron}.  A labeled abstract polyhedron where the image of $\Theta$ is
contained in the set $\{\pi/n \, \mid \, n \in \ZZ,\, n\geq 2\}$ is an
\textbf{abstract Coxeter polyhedron}.  An abstract Coxeter polyhedron $P$ gives
rise to an orientable $3$-orbifold $\QQ_{P}$ with base space $S^3$ and
singular locus consisting of a planar embedding of $P^{(1)}$.

A \textit{hyperbolic polyhedron} is the closure of a non--empty intersection
of finitely many open hyperbolic half--spaces.    There is a minimal
collection of half--spaces that determine the polyhedron.  The geodesic
planes in this minimal collection that bound the half--spaces are the
\textit{defining planes}.  Note that this definition allows for polyhedra of
infinite volume.

In the projective model of $\HH^3$, the defining planes extend to affine
planes in $\RR^3 \subset \RP^3$.  A \textit{vertex} of the polyhedron is a
point that is the intersection of $3$ or more of the extended defining
planes that lies in the intersection of the extended half--spaces in $\RP^3$.  A vertex is
said to be \textit{finite} if it lies in $\HH^3$, \textit{ideal} if it lies
in $S^2_{\infty} :=\overline{\HH}^3 \setminus \HH^3$, and \textit{hyperideal}
if it lies outside of $\overline{\HH}^3$.  A compact polyhedron has all
finite vertices.  A polyhedron with all ideal vertices is an \textit{ideal
polyhedron}.  
A \textit{hyperideal polyhedron} is a polyhedron that has at least one
hyperideal vertex.  A \textit{generalized polyhedron} is one where the
vertices may be finite, ideal, or hyperideal.

A labeled abstract polyhedron $(P,\Theta)$ is said to be \textit{realized by
$\PP$} if $\PP$ is a generalized hyperbolic polyhedron such that there is a
label--preserving cellular map between $(P,\Theta)$ and $\PP^{(1)}$, labeled
by dihedral angles.  A simple closed curve consisting of $k$ edges of $P^*$
is a \textit{$k$--circuit}, where $P^*$ is the dual graph to $P$.  If no two
edges in $P^*$ traversed by a $k$--circuit $\gamma$
are edges of a common face of $P^*$, then $\gamma$ is a \textit{prismatic $k$--circuit}.

Andreev's theorem gives necessary and sufficient conditions for a labeled
abstract polyhedron to be realizable as a hyperbolic polyhedron
\cite{andreev1,andreev2}.  An error in Andreev's proof was corrected by Roeder,
Hubbard, and Dunbar \cite{roeder}.  Hodgson also showed how Andreev's theorem
can be deduced from Rivin's characterization of convex hyperbolic polyhedra
\cite{hodand,rivpoly}.

\begin{theorem}[Andreev's theorem]\label{and}
A non--obtuse labeled abstract polyhedron $(P, \Theta)$ that has more than
$4$ vertices is realizable as a finite
volume hyperbolic polyhedron if and only if the following hold:
\begin{enumerate}
\item Each vertex meets $3$ or $4$ edges.
\item If $e_i,$ $e_j,$ and $e_k$ share a vertex then
$\Theta(e_i)+\Theta(e_j)+\Theta(e_k) \geq \pi$.
\item If $e_i,$ $e_j,$ $e_k,$ and $e_l$ share a vertex then
$\Theta(e_i)+\Theta(e_j)+\Theta(e_k)+\Theta(e_l) = 2 \pi$.
\item If $e_i,$ $e_j,$ and $e_k$ form a prismatic $3$--circuit, then 
$\Theta(e_i)+\Theta(e_j)+\Theta(e_k) < \pi.$
\item If $e_i,$ $e_j,$ $e_k,$ and $e_l$ form a prismatic $4$--circuit, then 
$\Theta(e_i)+\Theta(e_j)+\Theta(e_k)+ \Theta(e_l) < 2\pi.$
\item If $P$ has the combinatorial type of a triangular prism with edges
$e_i,$ $e_j,$ $e_k,$ $e_p,$ $e_q,$ $e_r$ along the triangular faces, then
$\Theta(e_i)+\Theta(e_j)+\Theta(e_k)+ \Theta(e_p)+\Theta(e_q)+\Theta(e_r) < 3\pi.$
\item If faces $F_i$ and $F_j$ meet along an edge $e_{ij}$, faces $F_j$ and
$F_k$ meet along an edge $e_{jk}$, and $F_i$ and $F_k$ intersect in exactly one
ideal vertex distinct from the endpoints of $e_{jk}$ and
$e_{ij}$, then $\Theta(e_{ij})+\Theta(e_{jk}) < \pi$.  
\end{enumerate}
Up to isometry, the realization of an abstract polyhedron is unique.  The
ideal vertices of the realization are exactly those degree $3$ vertices for
which there is equality in condition (2) and the degree $4$ vertices.
\end{theorem}

From this point forward, we assume that all vertices of abstract polyhedra
are of degree $3$ or $4$ unless we indicate otherwise.

The following is a generalization of Andreev's theorem that characterizes
generalized hyperbolic polyhedra.

\begin{theorem}\label{andbb}
Suppose that $(P,\Theta)$ is a non--obtuse labeled abstract polyhedron that
has more than $4$ vertices and is not a triangular prism.  Then $(P,\Theta)$
is realizable as a generalized hyperbolic polyhedron $\PP$ if and only if the
following conditions hold:
\begin{enumerate}
	\item If $e_i, e_j,$ and $e_k$ form a prismatic $3$--circuit, then 
		$\Theta(e_i)+\Theta(e_j)+\Theta(e_k) < \pi$,
	\item If $e_i, e_j, e_k,$ and $e_l$ form a prismatic $4$--circuit, then
		$\Theta(e_i)+\Theta(e_j)+\Theta(e_k)+\Theta(e_l) < 2\pi$, and
\item If faces $F_i$ and $F_j$ meet along an edge $e_{ij}$, faces $F_j$ and
$F_k$ meet along an edge $e_{jk}$, and $F_i$ and $F_k$ intersect in exactly one
ideal vertex distinct from the endpoints of $e_{jk}$ and
$e_{ij}$, then $\Theta(e_{ij})+\Theta(e_{jk}) < \pi$.  
\end{enumerate}
Moreover, a vertex of $\PP$ is finite, ideal or hyperideal if the link of the
vertex is spherical, Euclidean or hyperbolic respectively.  The polyhedron
$\PP$ has finite volume if and only there are no hyperideal vertices.
\end{theorem}

This theorem is a slight strengthening of Bao--Bonahon's characterization of
hyperideal polyhedra \cite{baobonahon} that is weaker than Andreev's theorem
for finite volume non--obtuse polyhedron.

\section{Algorithm detecting the geometric decomposition of polyhedral
orbifolds}\label{algbs}

Throughout this section, $(P,\Theta)$ will be a abstract Coxeter
polyhedron such that $P$ is trivalent and the sum of the labels given by
$\Theta$ to any three edges that share a vertex is greater than $\pi$.  Let
$\QQ_P$ be the compact orientable orbifold with base space equal to $S^3$ and
singular locus a planar embedding of $P$ with cone angles along edges of $P$
equal to twice the labeling given by $\Theta$.  The angle sum condition on
$\Theta$ ensures that $\QQ_P$ is a compact orbifold.  A labeling of the edges
of the dual graph, $P^*,$ is induced by labeling an edge of $P^*$ by the same
label as the corresponding edge in $P$.  

The following theorem follows from a theorem of Petronio that applies to
general $3$-orbifolds \cite{petronio}.  We
will give a simple proof in the case of polyhedral orbifolds.

\begin{theorem}[Petronio]\label{kmppolyorb}
	Let $P$ be an abstract Coxeter polyhedron.  Then there exists a unique
	spherical 2-suborbifold $\S$ of $\QQ_{P}$ such that each component of
	$\QQ_{P} \setminus \S$ with spherical boundary components capped off by
	orbifold balls is orbifold-irreducible.
\end{theorem}

After decomposing into orbifold-irreducible components, we will show how to
decompose along Euclidean $2$-suborbifolds into orbifold atoroidal pieces.
The existence of such a decomposition is implied by the splitting theorem of
Bonahon--Siebenmann \cite{bonahonsiebenmann}.  We we give a constructive proof of their theorem in the
setting of polyhedral orbifolds.  

\begin{theorem}[Bonahon-Siebenmann]\label{bspolyorb}
	Let $P$ be an abstract Coxeter polyhedron such that $\QQ_{P}$ is an
	orbifold-irreducible polyhedral orbifold.  Then there exists a Euclidean
	$2$-suborbifold $\TT$ of $\QQ_{P}$ such that each component of $\QQ_{P}
	\split \TT$ is either an orbifold Seifert fiber space or is orbifold
	atoroidal.  Furthermore, the set of atoroidal components of $\QQ_{P} \split
	\TT$ is canonical. 
\end{theorem}

An \textbf{orbifold Seifert fiber space} is a $3$-orbifold that fibers over
a $2$-dimensional orbifold such that each fiber has a neighborhood modeled on
$(D^2 \times S^1) \slash G$ where $G$ is a finite group that preserves both
factors of the product.  In the case of polyhedral orbifolds,
Proposition~\ref{sfpoly} which is proved in Section~\ref{S:sf} characterizes
orbifold Seifert fiber spaces.

Combining
our proofs of Theorems \ref{kmppolyorb} and \ref{bspolyorb} gives a
finite--time algorithm that produces the geometric decomposition of any
polyhedral orbifold in terms of the singular locus $P$ alone.

It should be noted that the techniques used in this section can be used to
find the geometric decomposition of any good $3$-orbifold with base space
$S^3$ and planar singular locus.  Any such singular locus must be
$2$-connected with the property that any pair of edges are labeled by the
same cone angle if there exists a simple closed curve in the plane that
intersects the singular locus in exactly those two edges.  Such a simple
closed curve corresponds to a prismatic $2$-circuit and each such circuit
corresponds to an incompressible spherical $2$-suborbifold with two cone
points of the same order.  It follows from a theorem of Cunningham-Edmonds
and the fact that the singular locus is a trivalent graph that there exists a
sequence of decompositions along prismatic $2$--circuits into abstract
polyhedra and bonds \cite{cunninghamedmonds}.  A \textbf{bond} is a graph that
consists of a pair of vertices joined by some number of edges.  If the
decomposition is applied to a trivalent graph with a Coxeter labeling, each
bond will be a pair of vertices joined by $3$ edges with angle sum greater
than $\pi.$  Each of these components is a spherical $3$-orbifold.
Theorem~\ref{kmppolyorb} holds without modification for non--compact
polyhedral orbifolds.  The proof of Theorem~\ref{bspolyorb}
may also be modified to work for non--compact polyhedral orbifolds by adding
a search for $(2,2,\infty)$ turnovers to the algorithm.

\subsection{Definitions}

A prismatic $3$-circuit $\gamma$ is said to be
\textbf{hyperbolic, Euclidean,} or \textbf{spherical} if sum of the labels along
the edges traversed by $\gamma$ is less than, equal to, or greater than $\pi$
respectively.  Similarly a prismatic $4$-circuit $\gamma$ is said to be
\textbf{hyperbolic} or \textbf{Euclidean} if the sum of the labels along
the edges traversed by $\gamma$ is less than $2\pi$ or equal to $2\pi$,
respectively.  This terminology reflects the fact that 
each prismatic $k$-circuit determines a $2$-suborbifold of $\QQ_{P}$ with the
specified geometry.  

If $\gamma$ is a prismatic $3$ or $4$-circuit in $P^*$, define $P^*$
\textbf{split along} $\gamma$, denoted $P^* \split \gamma$, as follows (see
Figure~\ref{split}):  First, form two new graphs $P^*_{\text{int}}$ and
$P^*_{\text{ext}}$ where $P^*_{\text{int}}$ consists of $\gamma$ along with
all edges and vertices interior to $\gamma$ with respect to a planar
embedding of $P^*$ and $P^*_{\text{ext}}$ consists of $\gamma$ along with all
edges and vertices exterior to $\gamma$. Let $\overline{P^*_{\text{int}}}$ be
the graph obtained by coning off the vertices of $\gamma$ in
$P^*_{\text{int}}$ to a vertex chosen to lie in the unbounded region of
$\RR^2 \setminus \gamma$ and $\overline{P^*_{\text{ext}}}$ to be the graph
obtained by coning off the $4$ vertices of $\gamma$ in $P^*_{\text{ext}}$ to
a vertex chosen to lie in the bounded region of $\RR^2 \setminus \gamma$.
Then $P^* \split \gamma$ consists of the disjoint union of
$\overline{P^*_{\text{int}}}$ and $\overline{P^*_{\text{ext}}}$.  Note that
$P^* \split \gamma$ is the union of the dual graphs of the components of
$\QQ_P \split \RRR(\gamma),$ where $\RRR(\gamma)$ is a $2$--suborbifold
realizing $\gamma$ and $\QQ_P \split \RRR(\gamma)$ denotes the closure of the
complement of $\RRR(\gamma)$ in $\QQ_P$. If $P$ is a polyhedral graph labeled by
$\theta$, then $P^* \split \gamma$ inherits a labeling that agrees with
$\theta$ on the original edges and equals $\pi/2$ on
the edges introduced by the splitting process.  The reader should note that
this procedure does not depend on the chosen planar embedding of $P^*$.

\begin{figure}
\labellist
\small\hair 2pt
\pinlabel $\delta$ [b] at 33 51
\pinlabel $\amalg$  at 140 31
\pinlabel $\amalg$  at 271 31
\pinlabel $P^*_{\text{ext}}$ [t] at 113 2
\pinlabel $P^*_{\text{int}}$ [t] at 170 2
\pinlabel $\overline{P^*_{\text{ext}}}$ [t] at 247 2
\pinlabel $\overline{P^*_{\text{int}}}$ [t] at 301 2
\pinlabel $\gamma$ [b] at 19 111
\pinlabel $\amalg$  at 140 113
\pinlabel $\amalg$  at 269 113
\pinlabel $P^*_{\text{ext}}$ [t] at 113 84
\pinlabel $P^*_{\text{int}}$ [t] at 170 84
\pinlabel $\overline{P^*_{\text{ext}}}$ [t] at 247 84
\pinlabel $\overline{P^*_{\text{int}}}$ [t] at 301 84
\endlabellist
\begin{center}
\scalebox{.9}{\includegraphics{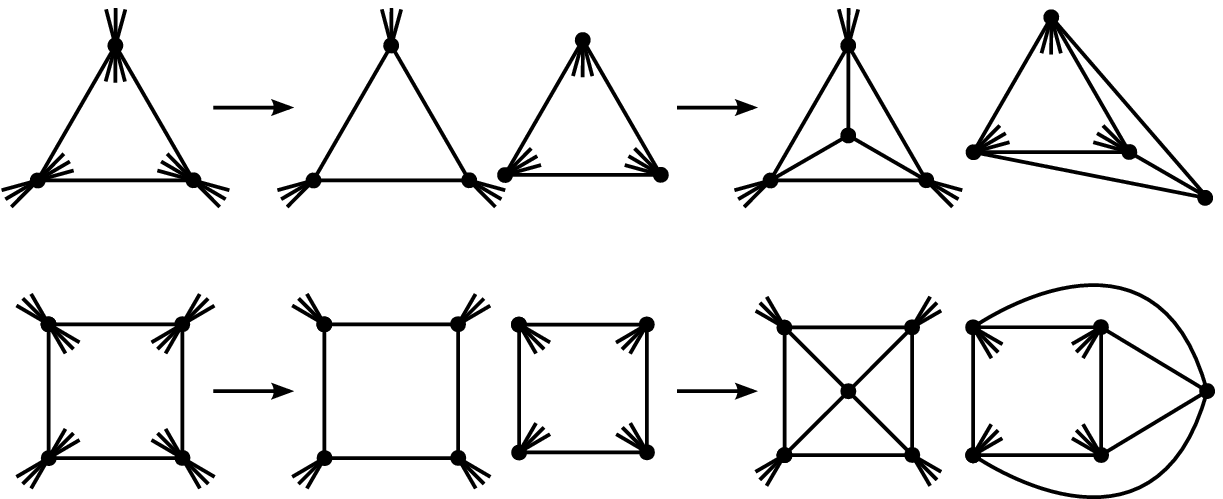}}
\end{center}
\caption{Schematic of splitting along a prismatic $3$-circuit $\gamma$ and a
prismatic $4$-circuit $\delta$}
\label{split}
\end{figure}

\subsection{Spherical decomposition}

A \textbf{turnover} is a $2$-orbifold of the form $S^2(p,q,r)$.  The notation
$S^2(p,q,r)$ indicates that the base space is $S^2$ and that the singular
locus consists of three cone points with cone angles $2\pi/p$, $2\pi/q$ and
$2\pi/r$.  A $2$-suborbifold $S$ of a $3$-orbifold $\QQ$ is
\textbf{incompressible} if either $\chi(S)>0$ and it does not bound an orbifold
ball in $\QQ$ or $\chi(S) \leq 0$ and any $1$-suborbifold on $S$ that bounds
an orbifold disk in $\QQ \setminus S$ bounds an orbifold disk in $S$.    A
$3$-orbifold $\QQ$ is said to be \textbf{orbifold irreducible} if every
spherical $2$-suborbifold bounds an orbifold ball.

\begin{lemma}
Every incompressible spherical $2$--suborbifold of $\QQ_{P}$ is a spherical turnover
that intersects $\Sigma(\QQ_{P})$ transversely in three edges with mutually
disjoint endpoints.
\end{lemma}

\begin{proof}
The fact that $S^3$ contains no incompressible spherical $2$--suborbifolds
implies that any such suborbifold $S$ must intersect the singular locus of
$\QQ_P$.  All spherical $2$--orbifolds have base space $S^2$ and either $0$,
$2$, or $3$ cone points.  The graph $P$ is $3$--connected, so any such
suborbifold must have $3$ cone points.  It also follows from
$3$--connectedness that any $2$--suborbifold $S$ that intersects $2$ edges
sharing a vertex $v$ also must intersect the third edge entering $v$.   Such
a $2$--suborbifold is compressible.
\end{proof}

\begin{proof}[Proof of Theorem~\ref{kmppolyorb}]
Any two prismatic $3$-circuits may be realized by disjoint $2$-suborbifolds
of $\QQ_P$.  Therefore, to construct $\S$, it suffices to take the collection
of spherical $2$-suborbifolds corresponding to the set of all spherical
prismatic $3$-circuits.  After capping off the boundary components of $\QQ_P
\setminus \S$, there are no spherical prismatic $3$-circuits.  This set is clearly unique. 
\end{proof}

\subsection{Definitions concerning $4$-circuits}

Let $\gamma$ be a Euclidean prismatic $4$-circuit with
vertices labeled cyclically by $v_1,$ $v_2,$ $v_3,$ and $v_4$.  Define the
\textbf{$1$-neighborhood} of $\gamma$, $N_1(\gamma)$,  to be the set of
Euclidean prismatic $4$-circuits that share vertices $v_1$ and $v_3$ with
$\gamma$.  Similarly,  define the \textbf{$2$-neighborhood} of $\gamma$,
$N_2(\gamma)$, to be the set of Euclidean prismatic $4$-circuits that share
vertices $v_2$ and $v_4$ with $\gamma$.  Note that $N_1(\gamma) \cap
N_2(\gamma) =
\gamma$.

The \textbf{support} of $N_i(\gamma)$ is the union of vertices and edges
traversed by elements of $N_i(\gamma)$.  Define the
\textbf{boundary} of $N_i(\gamma)$, denoted $\partial N_i(\gamma),$ to be the
set of $\delta \in N_i(\gamma)$ such that either $P^* \split \delta$ has a
component containing either no vertices of $\supp(N_i(\gamma))$ other than
those that are contained in $\delta$ or $P^* \split \delta$ has a component
containing exactly one vertex of $\supp(N_i(\gamma))$ that shares an edge
with each vertex of $\delta$ and consists of at least
$5$ triangles of $P^*$.  

A set $\{\gamma_1,
\gamma_2, \dots , \gamma_n\} \subset N_i(\gamma)$ is said to be
\textbf{admissible} if for each $i \neq j$, $\gamma_i$ is contained
completely in a single component of $P^* \split \gamma_j$.
A prismatic $4$-circuit $\gamma$ is said to be \textbf{trivial} if at least
one component of $\RR^2 \setminus \gamma$ contains exactly $1$ vertex of
$P^*$.  A \textbf{prism} is an abstract polyhedron that is graph isomorphic to the
$1$--skeleton of a polygon crossed with a closed interval.

\subsection{Seifert fibered polyhedral orbifolds}\label{S:sf}

In this section, we provide a complete classification of Seifert-fibered
polyhedral orbifolds.  

\begin{theorem}\label{sfpoly}
Suppose that $\QQ_P$ is a compact irreducible polyhedral orbifold.
Then $\QQ_P$ is orbifold Seifert-fibered if and only if
$P$ is a non-hyperbolic tetrahedron or a prism with labels $\pi/2$ along the
horizontal faces.
\end{theorem}

\begin{proof} 
If $P$ is a non-hyperbolic tetrahedron, necessity is immediate.  If $P$ is
a prism with labels $\pi/2$ along the horizontal faces, then it is a product
of a spherical, Euclidean, or hyperbolic polygon with an interval.  

Suppose that $\QQ_{P}$ is Seifert-fibered which implies that that $P$ is not
hyperbolic.  We will use the conditions in Andreev's theorem to show that the
singular locus of $\QQ_P$ must be as in the conclusion of the theorem.  The fact that $\QQ_{P}$ is irreducible implies that $P$ contains
no spherical prismatic $3$-circuits.  Suppose $P$ contains a Euclidean or
hyperbolic prismatic $3$-circuit $\gamma.$  Then at least $2$ of the edges
traversed by $\gamma$ will have labels strictly less than $\pi/2$.  By
proposition 2.41 of \cite{CHK}, these edges must actually be fibers of the
Seifert fibration.  The 2-suborbifold $C$ bounded by $\gamma$ is
incompressible, so must be either horizontal or vertical (See for example
Chapter 2 of \cite{hatch3}).  Since $C$ is transverse to the fibers
corresponding to the edges, $C$ must actually be horizontal.  This implies
that the third edge traversed by $\gamma$ is also a fiber.  Each face of $P$
containing a vertical fiber is covered by an incompressible $2$-suborbifold
in $\QQ_P$ that is vertical with respect to the Seifert fibration.  Each
vertical face is foliated by fibers, so must actually be a quadrilateral face
with the top and bottom edges labeled $\pi/2$.  It follows that $P$ must
actually be a triangular prism with top and bottom edges labeled $\pi/2$.  

If $P$ contains no spherical or Euclidean prismatic $3$ circuits and is not a
triangular prism, then $P$
must contain at least one Euclidean prismatic $4$-circuit in order to violate
Andreev's theorem. Each prismatic $4$-circuit may be realized as a
topological rectangle embedded in $P$.  Let $R$ be the collection of all such
rectangles, up to isotopy.  The rectangles in $R$ may be isotoped so that
pairwise they intersect transversely in arcs.  Let $\MN(R)$ be the union of a
closed regular neighborhood of the collection $R$ with any region of $P
\split R$ that intersects no edges of $P$.   Define $\partial \MN(R)$ to be the
$2$--suborbifold of $\MN(R)$ that is covered by the orbifold boundary of the
double cover of $\MN(R)$ in $\QQ_P$.

Suppose first that $r_1$, $r_2$, and $r_3$ are three rectangles in $R$ such
that $r_1 \cap r_2 \cap r_3$ is non-empty and such that no isotopy of $r_1$,
$r_2$, or $r_3$ leaves the intersection empty.  The rectangles then may be
further isotoped so that $r_1 \cap r_2 \cap r_3$ is a single point.  The
boundary of $\MN(r_1 \cup r_2 \cup r_3)$, viewed as a suborbifold, is the
disjoint union of eight right-angled triangles.  Because $\QQ_{P}$ is
irreducible, each of these triangles must actually bound orbifold balls in
$\QQ_{P}$, which implies that $R = r_1 \cup r_2 \cup r_3$ and that $\QQ_{P}$
is a Euclidean rectangular prism doubled along its boundary.

Now suppose that there are no triple points in $R$.  If $R$ contains $k$
rectangles, then the boundary of $\MN(R)$ consists of $2k$ rectangles.  The
remainder of the proof consists of proving that each complementary region of
$\MN(R)$ in $P$ has the combinatorics of a triangular prism with vertical
rectangular faces, as asserted by the conclusion of the proposition.  

Let $C$ be a component of $P \setminus \MN(R).$  We may think of
$C$ as a polyhedron with a rectangular face coming from $\partial{\MN(R)}$.
If $C$ has $5$ faces, then $C$ is a triangular prism and is either oriented
as desired, or rotated by a quarter turn.  If $C$ is the latter, this leads
to a contradiction for then $\Sigma(\QQ_{P})$ would contain a spherical prismatic
$3$-circuit.  The polyhedron $C$ is not hyperbolic, for this would
contradict the assumption that $\QQ_P$ is Seifert fibered.

However, $C$ is not a tetrahedron, contains no prismatic $3$-circuits, and
contains no prismatic $4$-circuits.  Hence for $C$ to violate Andreev's
theorem, it must actually be a prism with the edges of the triangular faces
labeled $\pi/2$.  This completes the proof.

\end{proof}

The following lemma indicates how to recognize prisms in terms of prismatic
$4$-circuits and their neighborhoods.

\begin{lemma}
If every vertex of $P^*$ is contained in $N_1(\gamma)$ for some $\gamma$,
then $P^*$ is dual to a prism.
\end{lemma}
\begin{proof}
If the assumption is satisfied, then $P^*$ consists of the vertices and edges
of $N_1(\gamma)$, along with the additional cycle of edges shown in
Figure~\ref{drumdual}.
\end{proof}

\begin{figure}
\labellist
\small\hair 2pt
\pinlabel $\dots$ at 96 46.5
\endlabellist
\begin{center}
\scalebox{.8}{\includegraphics{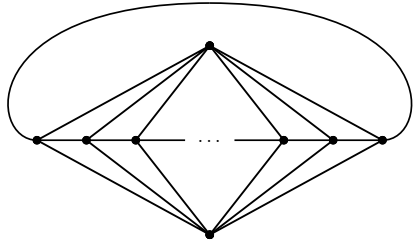}}
\end{center}
\caption{The dual graph to a prism}
\label{drumdual}
\end{figure}

\subsection{The algorithm}

If $P^*$ is the dual of an abstract polyhedron for $1 \leq i \leq k$,
define the \textbf{prismatic complexity} to be the $\NN$-valued function $c$
that assigns to $P^*$ the cardinality of the set
$$\mathcal{K}(P^*) = \{\delta \, \mid \, \delta \in \partial
N(\gamma), \text{ for some Euclidean prismatic } 4 \text{-circuit }
\gamma\}.$$
If $P_1^*, \dots,\,P_k^*$ are disjoint, extend $c$ by $c(\amalg_i P_i^*) =
\sum_i
c(P_i^*)$.

We may assume that $P^*$ contains no spherical prismatic
$3$-circuits by Theorem~\ref{kmppolyorb}.  We also may assume that all
Euclidean prismatic $3$--circuits are trivial by splitting $P^*$ along all
such $3$--circuits.

The decomposition algorithm goes as follows:

\begin{enumerate}
\item Set $P^*_0 = P^*$.  

\item While $c(P_k^*) > 0$,
\begin{enumerate}
\item If $P^*_k$ contains a nontrivial Euclidean prismatic $4$-circuit $\gamma$:  
\begin{enumerate}
\item If every vertex of the component $Q$ of $P^*_k$ containing $\gamma$
is contained in one of $N_1(\gamma)$ or $N_2(\gamma):$ 
\begin{enumerate}
\item Set
$P^*_{k+1} = P^*_k \setminus Q$ and record $Q$ in $\MC_{SF}$. 
\end{enumerate}
\item Else, set $P^*_{k+1} = P^*_k \split \MD$ where $\MD$ is a maximal
admissible subset of $\partial N_1(\gamma) \cup \partial N_2(\gamma)$.    
\end{enumerate}
\end{enumerate}
\begin{enumerate}
\item Else, if $P_k^*$ contains no nontrivial Euclidean prismatic
$4$-circuits: 
\begin{enumerate}
\item If a component $Q$ of $P_k^*$ contains a trivial Euclidean
prismatic $4$-circuit $\gamma$ and all vertices of $Q$ are contained 
$\supp(N_1(\gamma))$ or $\supp(N_2(\gamma))$: 
\begin{enumerate}
\item set $P^*_{k+1} = P_k^* \setminus Q$ and record $Q$ in $\MC_{SF}$.  
\end{enumerate}
\item Else, if $P_k^*$ contains no
nontrivial Euclidean prismatic $4$-circuits and each component of $P^*_k$ contains a
vertex not contained in the support of a neighborhood of a trivial Euclidean prismatic
$4$-circuit: 
\begin{enumerate}
\item set $\MC_{AT}$ equal to the disjoint union of the components of
$P^*_k$.  
\end{enumerate}
\end{enumerate}
\end{enumerate}
\item Return $\MC_{SF}$ and $\MC_{AT}$
\end{enumerate}

\begin{lemma}
	If $P^*$ is an abstract polyhedron containing no prismatic $3$--circuits
	and a nontrivial Euclidean prismatic $4$-circuit $\delta \in \partial
	N_i(\gamma)$ for some $\gamma$ and $i=1$ or $2$, then $c(P^*)> c(P^* \split \delta)$.
\end{lemma}

\begin{proof}

Suppose $\delta \in \K(P^*)$ is in $\partial N_i(\gamma)$ for some $\gamma$
and $i=1$ or $2$.  By definition, at least one component of $P^* \split
\delta$ contains either no vertices of $\supp( N_i(\gamma))$ or exactly one
vertex of $\supp( N_i(\gamma))$ that shares an edge with each vertex of
$\delta$ and is composed of at least $5$ triangles of $P^*$.  Let $Q^*$ be
such a component.  We will show that $\K(Q^*) \subseteq \K(P^*)$.  

If $\varepsilon \in \K(Q^*) \setminus \K(P^*)$, then $\varepsilon$ and
$\delta$ do not form an admissible pair of prismatic $4$--circuits in $Q^*$.
This implies that $Q^*$ must be a component of $P^* \split \delta$ that
contains at least one vertex, $v_5$.  By construction, we may choose $v_5$ to
share an edge with each vertex, $v_1$, $v_2$, $v_3$, and $v_4$ of $\delta.$
  Also, non--admissibility of
the pair $(\delta, \varepsilon)$ implies that $\varepsilon$ passes through
$v_5$.  The fact that $P^* \split \delta$ contains at least $5$ triangles of
$P^*$ implies that at least one of the triangles formed by $v_i$, $v_{i+1
(\mod 4)}$, and $v_5$ for $i \in \{1,2,3,4\}$ is a prismatic $3$--circuit.
This leads to a contradiction to irreducibility of $\QQ_{P}$ because at least
two edges of each of these triangles is labeled $2$.  This completes the
proof.

\end{proof}

\begin{corollary}
The algorithm terminates.
\end{corollary}

\begin{proof}
At each stage of the algorithm, the $\NN$-valued function $c$ decreases.  
\end{proof}

The following lemma follows trivially from the construction.

\begin{lemma}
Every Euclidean prismatic $4$--circuit in every component of $\MC_{AT}$ is
trivial.
\end{lemma}

The following proposition proves the final claim in Theorem~\ref{bspolyorb}
that the atoroidal components of the decomposition are canonical.  It is not
the case that $\MC_{SF}$ is independent of the choices made.

\begin{proposition}
The set $\MC_{AT}$ is independent of the choices made in the
algorithm.
\end{proposition}

\begin{proof}

If $\delta \neq \delta'$ are an admissible pair, then it is clear that
splitting along $\delta$ and splitting along $\delta'$ are commuting
operations $$(P^* \split \delta) \split \delta' = (P^* \split \delta') \split
\delta.$$

If two prismatic $4$-circuits $\delta$ and $\delta'$ in $\partial
N_1(\gamma)$ for some $\gamma$ are inadmissible as a pair then they each must
bound a region of the plane that contains exactly one vertex of
$N_1(\gamma)$.  If necessary, choose a new embedding of $P^*$ into the plane
so that both of these regions are bounded.  Then, the configuration of
$\delta$ and $\delta'$ must be as in the topmost diagram in
Figure~\ref{ambiguity}.  With labels as in the figure, if $\delta$ and
$\delta'$ form an inadmissible pair, then $u$ and $w$ must be joined by the
edge and the embedding may be chosen so that region bounded by the
$4$-circuit passing through $u$, $v_1$, $v_3$ and $w$ must contain at least
one vertex that is in $P^* \setminus N_1(\gamma)$.  The other two bounded
regions contain no vertices of $N_1(\gamma)$ but can be otherwise arbitrarily
chosen.

\begin{figure}
\labellist
\small\hair 2pt
\pinlabel $\approx$ at 207 25
\pinlabel $\amalg$ at 133 196 
\pinlabel $\amalg$ at 299 196 
\pinlabel $\amalg$ at 66 111 
\pinlabel $\amalg$ at 137 111 
\pinlabel $\amalg$ at 288 111 
\pinlabel $\amalg$ at 360 111 
\pinlabel $\amalg$ at 66 25 
\pinlabel $\amalg$ at 133 25
\pinlabel $\amalg$ at 293 25
\pinlabel $\amalg$ at 358 25
\pinlabel $\delta$ [br] at  190 306
\pinlabel $\delta'$ [bl] at 226 306
\pinlabel $\split \delta$ [br] at 156 250 
\pinlabel $\split \delta'$ [bl] at 263 250 
\pinlabel $u$ [t] at 194 283
\pinlabel $w$ [t] at 223 283
\pinlabel $v_1$ [b] at 209 329
\pinlabel $v_3$ [t] at 209 244
\endlabellist
\begin{center}
\scalebox{.83}{\includegraphics{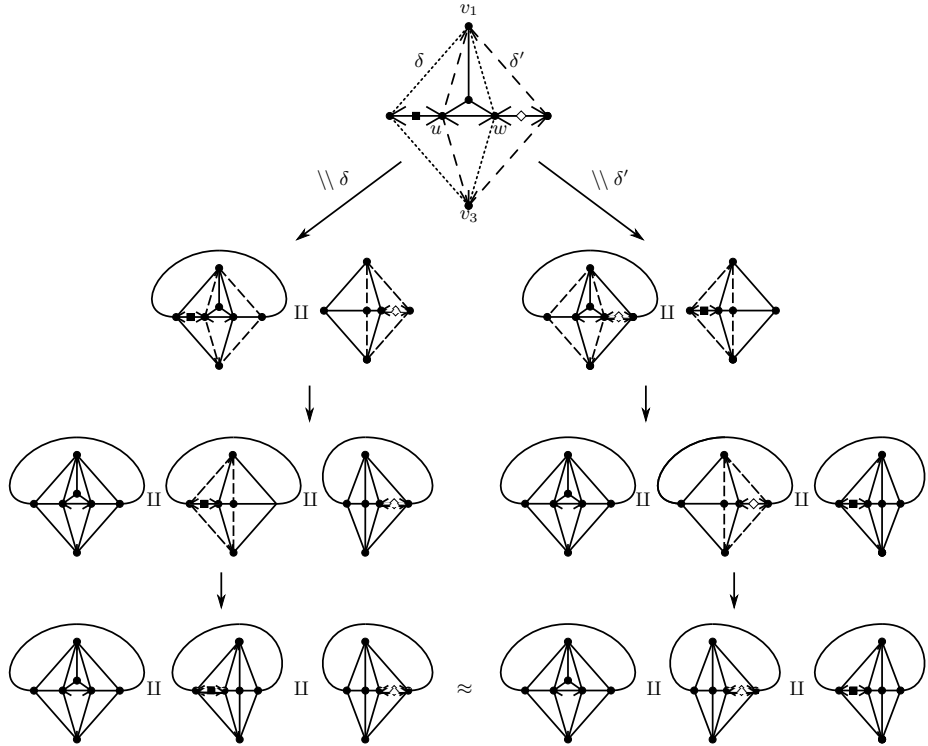}}
\end{center}
\caption{In the topmost figure, the short--dashed circuit is $\delta$ and the long--dashed circuit
is $\delta'$.  After the first splitting, the algorithm says to split
successively along the dashed curves in the second and third rows.  The
bottommost diagrams are graph isomorphic.   The Seifert fiber components that
are produced are not shown in
this figure}
\label{ambiguity}
\end{figure}

The remainder of Figure~\ref{ambiguity} shows that choice of splitting first
along $\delta$ yields the same atoroidal components as first splitting along
$\delta'$.  The reader should note that the Seifert fiber components
produced do not agree.
\end{proof}

\section{A decomposition for non--obtuse hyperbolic polyhedra}\label{C:decomp}

In this section, we show how to apply the decompositions of Petronio and
Bonahon-Siebenmann described in the previous section to decompose a
non-obtuse hyperbolic polyhedron into components that remain hyperbolic upon
being relabeled by $\pi/2$ and the complement of these components.  The
complementary pieces will generally be hyperbolic cone manifolds with
non--geodesic boundary.  These components will be discussed more thoroughly
in Section~\ref{C:prisms}.

Suppose that $\PP$ is a non--obtuse hyperbolic polyhedron that realizes a
labeled abstract polyhedron $(P,\Theta)$.  Let $\QQ_{\PP}$ be the cone
manifold obtained by doubling $\PP$ along its boundary.  In the case where
$\PP$ is a Coxeter polyhedron, $\QQ_{\PP}$ is an orbifold with fundamental
group equal to the index--$2$ orientation preserving subgroup of the
reflection group generated by $\PP$.  It will be useful to consider the
associated compact topological orbifold $\QQ_{\PP}^{\perp}$ that is obtained
from $\QQ_{\PP}$ by changing all cone angles to $\pi$ and capping off each of
the punctures in $S^3$ that correspond to degree $4$ ideal vertices with
pillowcases and each of the punctures that correspond to degree $3$ ideal
vertices with orbifold balls.  These pillowcases are then part of the
boundary of $\QQ_{\PP}^{\perp}$.  Equivalently one can consider the
topological closure of the non--compact orbifold.  This procedure is
analogous to passing from a finite volume hyperbolic manifold, $M,$ to a
compact topological manifold $\overline{M}$ by truncating the cusps or
forming the closure of $M$.

\subsection{Turnover decomposition}\label{S:turndecomp}

The results in this section are a constructive version of a theorem of Dunbar
that imply that a non--obtuse hyperbolic polyhedron may be decomposed along a
disjoint union of hyperbolic turnovers into components that contain no
nontrivial prismatic $3$-circuits \cite{dunbar}.  In the case of hyperbolic polyhedral
orbifolds, the collection of turnovers produced by Dunbar's theorem
corresponds to the collection of turnovers that pass through the same edges
of $\PP$ as the spherical turnovers produced by Theorem~\ref{kmppolyorb}
applied to $\QQ_{\PP}^{\perp}$.

We generalize our earlier definition of turnovers to allow for cone--manifold
type singularities.  That is, we define a \textbf{turnover} to be a
$2$--dimensional cone manifold obtained by doubling a triangle with angles
$\alpha,$ $\beta,$ and $\gamma$ along its boundary.    A turnover is
hyperbolic, Euclidean, or spherical if $\alpha + \beta + \gamma$ is less
than, equal to, or greater than $\pi$, respectively.  

For each of turnover $s_i$ in the collection $\S$ produced by
Theorem~\ref{kmppolyorb} applied to $\QQ_{\PP}^{\perp}$, there is an associated
turnover $t_i$  in $\QQ_{\PP}$ that intersects the same three edges of the
singular locus that $s_i$ intersects.  The rest of this section will show
that the collection of turnovers in $\S$ can be constructed directly from the
abstract polyhedron $P$ and that there is a geodesic representative in the
isotopy class of each $t_i$ associated to an $s_i \in \S$.

In the projective model of $\HH^3$, a geodesic plane is the
intersection of the open unit ball with an affine plane in $\RP^3$.  Suppose
that $v\in \RP^3$ is a point not contained in $\overline{\HH}^3$.  Consider
the set of affine lines that pass through $v$ and are tangent to
$\partial_{\infty}\HH^3$.  The intersection of this set of lines with the
boundary of $\HH^3$ is a circle.  The intersection of the plane containing
this circle with $\HH^3$ is the \textit{polar hyperplane of $v$.}  Any
hyperbolic geodesic that extends to a line passing though $v$ is orthogonal
to the polar hyperplane of $v$.

The following lemma says that to find embedded turnovers in $\QQ_{\PP}$, it
suffices to find prismatic $3$--circuits in $P$.  

\begin{lemma}\label{turn}
If $\gamma \subset P^*$ is a prismatic $3$--circuit, then there exists a
unique hyperbolic turnover $t$ embedded in $\QQ_{\PP}$ that meets the faces
of $\PP$ through which $\gamma$ passes orthogonally.
\end{lemma}

\begin{proof}
We will work in the projective model of $\HH^3$.  Consider the three defining
planes $\Pi_1,$ $\Pi_2$, and $\Pi_3$ in which the faces of $\PP$
corresponding to the vertices of $\gamma$ lie.  By Andreev's theorem
$\Theta(e_{12})+\Theta(e_{23})+\Theta(e_{13})<\pi$ where $e_{ij}=\Pi_i \cap
\Pi_j,$ so $\Pi_1 \cap \Pi_2 \cap \Pi_3 = v$ is a point in $\RP^3 \setminus
\overline{\HH}^3$.  The polar hyperplane, $\Pi_v$, of $v$ is orthogonal to
$e_{12},$ $e_{23},$ and $e_{13}$, hence is orthogonal to $\Pi_1,$ $\Pi_2$,
and $\Pi_3$.  Since $\PP$ is non--obtuse, the intersection of $\Pi_v$ with the
three half spaces determined by the $\Pi_i$ that contain $\PP$ is actually
contained in $\PP$.  Therefore, the double of $\Pi_v \cap \PP$ is the
desired hyperbolic turnover, $t$.  
\end{proof}

\begin{lemma}\label{turndis}
If $\gamma_1 \neq \gamma_2$ are prismatic $3$--circuits, then
the associated turnovers $t_1$ and $t_2$ provided by Lemma~\ref{turn} are
disjoint.  
\end{lemma}

\begin{proof}
Suppose for contradiction that $t_1 \cap t_2 \neq \emptyset.$  If $t_1
\cap t_2$ is a single point, $p$, then $p$ must lie in an edge of the
polyhedron $\PP$.  By the previous lemma, both $t_1$ and $t_2$ intersect the
$1$--skeleton of the polyhedron orthogonally.  Hence the two turnovers actually
coincide.  The only other possibility is that $t_1 \cap t_2$ is
$1$--dimensional.  Both turnovers are geodesic, so in $\QQ_{\PP}$ the
intersection is a closed geodesic.  This is a contradiction as hyperbolic
turnovers contain no closed geodesics.  
\end{proof}

A polyhedron is said to be \textit{turnover reduced} if every prismatic
$3$-circuit is trivial.  If $\PP$ is a turnover reduced hyperbolic
polyhedron, then $\QQ_{\PP}^{\perp}$ is orbifold irreducible.  The following is
a corollary of Theorem~\ref{kmppolyorb} and Lemmas \ref{turn} and
\ref{turndis}.

\begin{corollary}\label{turnredlemma}
For any non--obtuse hyperbolic polyhedron $\PP$, there exists a finite
collection $\S$ of disjoint, embedded, nonparallel turnovers  such that
the closure of each component of $\PP \setminus \S$ is a turnover
reduced non--obtuse hyperbolic polyhedron.  
\end{corollary}

\subsection{Atoroidal components of quadrilateral
decomposition}\label{atorcpts}

In this section we show that the components of $\PP$ that correspond to atoroidal
components of $\QQ_{\PP}^{\perp}$ coming from the decomposition of
Theorem~\ref{bspolyorb} have hyperbolic interiors.  We first explain more
explicitly the connection between the prismatic $4$--circuits produced by
Theorem~\ref{bspolyorb} and the $2$--suborbifolds along which
$\QQ_{\PP}^{\perp}$ is decomposed.

A incompressible suborbifold of the form $S^2(2,2,2,2)$ embedded in
$\QQ_{\PP}^{\perp}$ is a topological sphere embedded in
the base space of $\QQ_{\PP}^{\perp}$ that intersects the singular locus in four
edges that form a prismatic $4$--circuit.  The singular locus,
$\Sigma(\QQ_{\PP}^{\perp})$, is contained in a $2$--sphere, topologically
embedded in the base space.  We may assume that $S^2(2,2,2,2)$ has been
isotoped so that it intersects this $2$--sphere transversely.  The
$3$--orbifold $\QQ_{\PP}^{\perp}$ admits an order--two self--homeomorphism
that
fixes the $2$--sphere in which $\Sigma(\QQ_{\PP}^{\perp})$ is embedded and
swaps the complementary components.  Let $\PP^{\perp}$ be the quotient of
$\QQ_{\PP}^{\perp}$ by the action of this symmetry.  The quotient,
$\PP^{\perp}$, is a non--orientable $3$--orbifold with boundary consisting of
the quadrilaterals that come from the quotient of the bounding pillowcases
of $\QQ_{\PP}^{\perp}$ by the action.  The incompressible $S^2(2,2,2,2)$
suborbifolds descend to embedded quadrilaterals in $\PP^{\perp}$.
Theorem~\ref{bspolyorb} then leads to a decomposition of $\PP^{\perp}$ by
quadrilaterals into components that correspond to the atoroidal and
Seifert--fibered components of the double cover.  The decomposition of
$\PP^{\perp}$ by quadrilaterals leads to a decomposition of the original
hyperbolic polyhedron $\PP$ by quadrilaterals.    The boundary of a
component of the decomposition of $\PP^{\perp}$ is the union of the
decomposition quadrilaterals that it meets.

The following proposition says that the atoroidal components coming from the
quadrilateral decomposition admit hyperbolic structures.

\begin{proposition}\label{topdecomp}
Let $Q$ be a component of the decomposition of $\PP$ corresponding to an
atoroidal component of the decomposition of $\PP^{\perp}$.  Let $R$ be the
abstract polyhedron with a quadrilateral or triangular face for each
quadrilateral or triangular boundary component of $Q$.  Then $(R,\Theta)$ is
realizable as a hyperbolic polyhedron where $\Theta$ is the labeling which
agrees with the dihedral angles of $\PP$ and assigns $\pi/2$ to each of the
introduced edges.
\end{proposition}

\begin{proof}
The proof consists of showing that $(R,\Theta)$ satisfies the conditions of
Andreev's theorem.  The link of each introduced vertex will be spherical
because each such vertex meets at least two edges with dihedral angle
$\pi/2.$  

Suppose that the introduction of a quadrilateral face, $q$, created a prismatic
$3$--circuit $\gamma$ that passes through edges $e_1$ and $e_2$ of $q$ along with
an edge $e_3$ not in $q$.  By assumption, $Q$ is turnover--reduced, so
$\gamma$ must be parallel to a triangular face.  This is a contradiction
because in $\QQ_{\PP}^{\perp},$ the $4$--circuit corresponding to $q$ would
not be prismatic.

No prismatic $4$--circuits pass through any of the introduced faces, for this
would contradict the fact that $Q$ corresponds to an atoroidal component of
the decomposition.
\end{proof}

  We use the following
theorem of Agol, Storm, and W. Thurston to show that the procedure in
Proposition~\ref{topdecomp} does not increase the volume of the atoroidal
components \cite{ast}.

\begin{theorem}[Agol--Storm--W. Thurston]\label{ast}
Let $\overline{M}$ be a compact manifold with interior M, a hyperbolic
$3$--manifold of finite volume.  Let $\overline{\Sigma}$ be an incompressible
surface in $\overline{M}$.  Then
$$\vol(M) \geq \frac{1}{2} V_3 || D( M \split \Sigma) ||,$$
where $D( M \split \Sigma)$ denotes the double of $M \split \Sigma$ along
$\Sigma$ and $|| \cdot ||$ denotes the Gromov invariant.
\end{theorem}

In the case where $\Sigma$ is separating and each component of $M
\split \Sigma$ is hyperbolic, this theorem says that the sum of the volumes of
the rehyperbolized components of $M \split \Sigma$ is no more than the volume of
$M$.

The following shows that the volume of an atoroidal component of the Bonahon--Siebenmann
decomposition of $\QQ_{\PP}^{\perp}$ is no less than that of the
corresponding component with totally geodesic boundary, as described in
Proposition~\ref{topdecomp}.

\begin{proposition}
Let $Q$ be a component of the decomposition of $\PP$ corresponding to an
atoroidal component of the decomposition of $\QQ_{\PP}^{\perp}$ and let $\MRR$
be the realization of $(R,\Theta)$ as described in
Proposition~\ref{topdecomp}.  Then $\vol(\MRR) \leq \vol(Q)$.
\end{proposition}

\begin{proof}
Consider the double, $D(\QQ_Q)$ of the orbifold $\QQ_Q$ along its boundary.  The
orbifold, $D(\QQ_Q)$, admits a finite volume hyperbolic structure. 
By Selberg's lemma, there exists a finite index manifold cover $M_Q$ of
$D(\QQ_Q)$.  The preimage of $\partial(D(\QQ_Q))$ in $M_Q$ is an
incompressible surface $\Sigma$.  If the index of the cover is $n$, then
$\vol(M_Q) = n \vol(D(\QQ_Q)) = 4n\vol(Q).$  Then Theorem~\ref{ast} implies
	
$$4n \vol(Q) =  \vol(M_Q) \geq \frac{1}{2} V_3 ||D(M \split \Sigma)|| 
= 4n\vol(\MRR).$$
\end{proof}

\section{Deformations of polyhedra and volume change}\label{deformation}

Suppose that $P$ is an abstract polyhedron with $E$ edges.  Labelings of $P$ are
given by points in $\RR^E$.  Throughout, we consider only
non--obtuse labelings.  Define 
$$\Omega(P)=\left\{ \Theta \in (0,\pi/2]^E \, \left\vert \, 
\begin{array}{c}
 (P,\Theta)\text{ is realizable as a finite volume}  \\
 \text{hyperbolic polyhedron} 
\end{array}
\right. \right\}.$$
This set is in one--to--one correspondence with the set of isometry classes of
non--obtuse hyperbolic polyhedra of finite volume with $1$--skeleton
isomorphic to $P$. For convenience will pass
from labelings to polyhedra without comment.

If $\Omega(P)$ is non--empty, Andreev's theorem implies that the
closure of $\Omega(P)$ is a convex polytope in $\RR^E$.  Define $\whop(P)$
similarly as the set of non--obtuse labelings that yield a hyperbolic polyhedron of finite
or infinite volume. Theorem~\ref{andbb} implies that the closure of $\whop(P)$ is also a
convex polytope in $\RR^E$. 

Suppose that $\PP(\Theta_0)$ and $\PP(\Theta_1)$ are hyperbolic realizations of
labeled abstract polyhedra $(P,\Theta_0)$ and $(P,\Theta_1)$.  A \textit{smooth
deformation from $\PP(\Theta_0)$ to $\PP(\Theta_1)$} is a piecewise--smooth map
$$\Phi : [0,1] \longrightarrow \whop(P)$$
such that $\Phi(0) = \PP(\Theta_0)$ and $\Phi(1) = \PP(\Theta_1).$
A deformation $\Phi$ is said to be  \textit{angle--nondecreasing} or
\textbf{angle--nonincreasing} if the
projection to each coordinate of the target space  composed with $\Phi$ is a
nondecreasing or nonincreasing function, respectively.

The following proposition says that there exists an angle--nondecreasing
deformation from any generalized hyperbolic Coxeter polyhedron with no prismatic
$3$--circuits to a finite volume hyperbolic polyhedron with all dihedral
angles $\pi/2$ or $\pi/3$.

\begin{proposition}\label{unif}
Suppose $P$ is an abstract polyhedron with at least $6$ faces containing no
prismatic $3$--circuits and $\Theta \in \whop(P)$ is of the form
$\Theta(e_i)=\pi/n_i,$ where each $n_i \geq 2$ is an integer.  Then there
exists $\Theta' \in \Omega(P)$ of the form $\Theta'(e_i)=\pi/m_i,$ where each
$m_i \in \{2,3\}$ and $m_i \leq n_i.$ 
\end{proposition}

\begin{proof}
Define $\Theta'$ as follows:
\begin{equation}
	\Theta'(e_i)=
		\begin{cases}
		\pi/n_i & \text{if }n_i=2 \text{ or } 3 \\
		\pi/3 & \text{if }n_i>3.
		\end{cases}
\end{equation}
It suffices to check that $\Theta'$ satisfies the conditions Andreev's theorem
(Theorem~\ref{and}).  By assumption, conditions (1), (2), (3), (4), and (6) are
satisfied.   
The labeling $\Theta$ satisfies condition (2) of the hyperideal version of
Andreev's theorem (Theorem~\ref{andbb}), so for
any prismatic $4$--circuit formed by edges $e_p$, $e_q$, $e_r$, and $e_s$,
$\Theta(e_i)\leq \pi/3$ for at least one $i\in\{p,q,r,s\}$.  For such an $i$,
$\Theta'(e_i)=\pi/3$.  Hence
$\Theta'(e_p)+\Theta'(e_q)+\Theta'(e_r)+\Theta'(e_s)<2\pi$, so $\Theta'$
satisfies condition (5).  The argument is
similar to show that $\Theta'$ satisfies condition (7).
\end{proof}

Schl\"{a}fli's formula describes how the volume of a polyhedron changes as it
is deformed.  The following generalization of Schl\"{a}fli's formula is due
to Milnor.  Rivin supplies a proof in \cite{rivineucl}.

\begin{theorem}[Schl\"{a}fli's formula]
Let $\PP$ be a non--obtuse hyperbolic polyhedron with vertices $v_1, \dots,
v_n,$ where $v_i$ is ideal for $i>m.$  Let $H_{m+1}, \dots, H_{n}$ be a
collection of horospheres such that $H_i$ is centered at $v_i$.  Further if
there is an edge $e_{ij}$ between $v_i$ and $v_j$, then $l_{ij}$ is the
distance between $v_i$ and $v_j$ if $i,j \leq m$, the signed distance between
$H_i$ and $H_j$ (negative if the corresponding horoballs intersect) if $i,j >
m$ and the signed distance between $H_i$ and $v_j$ if $i \leq m$ and $j>m.$
Then $$d\vol(\PP) = -\frac{1}{2} \sum_{\text{edges } e_{ij}} l_{ij}
d\theta_{ij},$$ where $\theta_{ij}$ is the dihedral angle along the edge
$e_{ij}.$
\end{theorem}

The perhaps surprising fact that the right--hand side does not depend on the
choice of horoballs follows from the fact that the link of an ideal vertex is
a
Euclidean polygon.

One consequence of Schl\"{a}fli's formula is that angle--increasing
deformations of polyhedra are volume--decreasing.  In particular, it implies
that the deformation given in Proposition~\ref{unif} is volume--nonincreasing.

\begin{corollary}\label{aicor}
Suppose that $P$ is an abstract polyhedron with no prismatic $3$--circuits.
Let  $\Theta, \Theta' \in \Omega(P)$ be as in Proposition~\ref{unif}.
If $\PP$ and $\PP'$ are the hyperbolic realizations of $(P,\Theta)$ and
$(P,\Theta')$ respectively, then $\vol(\PP) \geq \vol(\PP')$ with equality if and
only if $\Theta=\Theta'$.
\end{corollary}

\begin{proof}  
By Theorem~\ref{and}, $\Omega(P)$ is a convex polytope in $\RR^E$.  Hence the
line segment 
$$s(t)=(1-t)\Theta+t\Theta'$$ 
is contained in $\Omega(P)$.  Let $\PP_t$ be the hyperbolic realization of
$(P,s(t))$ for $t\in[0,1]$.  By definition of $\Theta',$ the labeling
$s(t)$ restricted to each edge is nondecreasing in $t$.  Schl\"{a}fli's formula
then implies that $\vol(\PP)\geq \vol(\PP').$  Since the volume is
nonincreasing along $s(t)$, $\Theta=\Theta'$  if and only if
$ds = 0$, in which case, the volume is constant.  
\end{proof}

We will now describe how to deal with turnover reduced polyhedra, that is,
polyhedra in which every prismatic $3$--circuit is trivial.  Suppose $\PP$ is
a generalized hyperbolic polyhedron.  Define the \textit{truncation of $\PP$}, denoted
$\PP^{\vee}$, to be the polyhedron
defined by the same planes as $\PP$ along with the polar hyperplanes of any
hyperideal vertices.  If $\PP$ has no hyperideal vertices, then
$\PP^{\vee}=\PP$. If $\Theta$ is a labeling of an abstract polyhedron $P$
realized by $(\PP,\Theta)$, then $P^{\vee}$ has an induced labeling
$\Theta^{\vee}$ defined by keeping all edge labels the same and labeling the
introduced edges $\pi/2$.

Suppose that $P$ is an abstract polyhedron that contains a prismatic
$3$--circuit that is parallel to a triangular face $T$.  The
\textit{extension} of $P$, denoted $\widetilde{P}$, is obtained by replacing
all triangular faces of $P$ that are parallel to prismatic $3$--circuits by
vertices.  The three edges that formed the prismatic $3$--circuit in $P$ are
incident to the new vertex in $\widetilde{P}$.  Geometrically, extension is
roughly inverse to truncation.  More explicitly, suppose that
$\Theta\in\Omega(P)$, $\widetilde{\Theta}$ is the restriction of $\Theta$ to
$\widetilde{P}$ and $\PP$ and $\widetilde{\PP}$ are the hyperbolic
realizations of $(P,\Theta)$ and $(\widetilde{P},\widetilde{\Theta})$
respectively.  If $\Theta$ assigns $\pi/2$ to each of the edges of a
triangular face $T$, then $\PP=\widetilde{\PP}^{\vee}$.  If $\Theta$ assigns an
angle of less than $\pi/2$ to any of the edges in $T$, then $\PP$ contains a
polyhedron $\QQ$ that differs from $\PP$ by a collection of triangular
prisms or tetrahedra such that $\QQ=\widetilde{\QQ}^{\vee}.$  Note that if a
polyhedron has the property that any prismatic $3$--circuit is parallel to a
face, then the extension of such a polyhedron has no prismatic $3$--circuits
whatsoever.  

A deformation \textit{with face degenerations} is a deformation
of a polyhedron in which a face degenerates to a vertex of any type.  In the
case of a polyhedron with a triangular face that is parallel to a prismatic
$3$--circuit, face degeneration occurs as the angle sum along the prismatic
$3$--circuit approaches and possibly exceeds $\pi$.  

\begin{corollary}\label{pi3deform}
If $\PP$ is a turnover reduced hyperbolic polyhedron realizing $(P, \Theta)$
with $\Theta(e)\leq \pi/3$ for any edge $e$ that is not an edge of a
triangular face, then there exists an angle--nondecreasing deformation with
face degenerations
to a $\pi/3$--equiangular ideal polyhedron $\PP'$.
\end{corollary}

\begin{proof} First note that $\wt{\PP}^{(1)}$ is graph isomorphic to the
graph $P'$, obtained from 
$P$ by replacing all triangular faces that are parallel to prismatic
$3$--circuits by vertices.  The fact that $\Theta \in \Omega(P)$
implies that $\wt{\Theta} \in \whop(P').$  An application of
Proposition~\ref{unif}
yields the desired labeling of $P'$.  \end{proof}

A polyhedron is \textit{atoroidal} if every prismatic $4$--circuit is
parallel to a face.

\begin{corollary}\label{pi2deform} If $\PP$ is a turnover reduced and
atoroidal  hyperbolic polyhedron realizing $(P, \Theta)$ with $\Theta(e) =
\pi/2$ for any edge $e$ that is part of a triangular or rectangular face
with all degree $3$ vertices, then there exists an angle--nondecreasing
deformation with face degenerations to a right-angled polyhedron $\PP'$.  \end{corollary}

\begin{proof}
The argument is similar to the previous corollary.
\end{proof}

For an abstract polyhedron $P$, define $$V \colon \whop(P) \to \RR$$ by
$V(\PP,\Theta)=\vol_{\HH^3}(\PP^{\vee},\Theta^{\vee}).$  A generalization of Milnor's
continuity conjecture by Rivin implies that $V$ is continuous on
$\whop(\PP)$ \cite{rivincont}.  Hence by Schl\"{a}fli's formula, the
deformations in Corollaries \ref{pi3deform} and \ref{pi2deform} are volume
nonincreasing.

A proof of the lower bound in Theorem~\ref{superweak} follows from
Corollary \ref{pi2deform} and a theorem from \cite{volpoly} that we restate
here for convenience:  

\begin{theorem}[\cite{volpoly}] 
\label{pi2general} If $\PP$ is a right--angled hyperbolic polyhedron,
$N_{\infty}$ ideal vertices and $N_F$ finite vertices, then 
$$\vol(\PP) \geq \frac{4N_{\infty}+N_F-8}{32} \cdot V_8.$$ 
\end{theorem}

The following
corollary is a better lower bound than that in Theorem~\ref{superweak}, that
follows by disregarding the contributions of the prismatic $3$-circuits.

\begin{corollary}
Let $\PP$ be a non--obtuse hyperbolic polyhedron containing
no prismatic $4$--circuits, $N_4$ degree $4$ vertices, $N_3$ degree $3$
vertices, and $M_3$ prismatic $3$--circuits. Then
	$$ \vol(\PP) > \frac{4N_4 + (N_3+M_3) -8}{32} \cdot V_8.$$
\end{corollary}

\begin{proof}
If $\PP$ contains no prismatic $4$-circuits, then by
Corollary~\ref{pi2deform}, there exists a volume--decreasing deformation from
$\PP$ to a right--angled polyhedron with $N_4$ degree--$4$ vertices and
$N_3+M_3$ degree $3$ vertices.  Theorem 2.4 in \cite{volpoly} gives the
conclusion.
\end{proof}

\section{On hyperbolic prisms and their volumes}\label{C:prisms}

In this section a lower bound on the volume of a hyperbolic Coxeter prism is
produced by exhibiting the minimal volume Coxeter $n$--prism for each $n\geq
5$.
This lower bound does not extend to the full non--obtuse case, but provides a
lower bound for any non--obtuse prism having no dihedral angles in the
interval $(\pi/3,\pi/2)$. The final subsection in this section gives a lower
bound on the volume of the components of a hyperbolic Coxeter polyhedron
$\PP$ that correspond to the Seifert--fibered components coming from the
decomposition of $\QQ_{\PP}^{\perp}$ given by Theorems~\ref{kmppolyorb} and
\ref{bspolyorb}.  Again, the results of the final subsection extend to the
case of non--obtuse prisms with no dihedral angles in the interval $(\pi/3,
\pi/2)$.

An \textit{$n$--prism} is a non--obtuse hyperbolic polyhedron consisting of $2$
disjoint $n$--gon faces and $n$ quadrilateral faces as shown in
Figure~\ref{prism}.  Label the edges of one of the $n$--gon faces cyclically by
$a_1, a_2 \dots a_n$ and the edges of the other $n$--gon by $b_1,b_2 \dots
b_n$ so that $a_i$ and $b_i$ are edges of the same quadrilateral face.  Label
the remaining $n$ edges by $c_1,c_2 \dots c_n$ so that $c_i$ is an edge of
the quadrilateral faces containing $a_i$ and $a_{i+1}$, where the labeling is
taken modulo $n$. See Figure~\ref{prism}.   Label the dihedral angles along
the edges $a_i$, $b_i$ and $c_i$ by $\aa_i$, $\bb_i$ and $\cc_i$,
respectively.  

\begin{figure}
\labellist
\small\hair 2pt
\pinlabel $a_1$ [br] at 48 130
\pinlabel $a_2$ [bl] at 27 91
\pinlabel $a_3$ [b] at 84 61
\pinlabel $a_4$ [br] at 155 69
\pinlabel $a_5$ [l] at 183 111
\pinlabel $a_6$ [b] at 128 142
\pinlabel $b_1$ [tl] at 54 85
\pinlabel $b_2$ [tr] at 24 45
\pinlabel $b_3$ [t] at 80 12
\pinlabel $b_4$ [tl] at 163 23
\pinlabel $b_5$ [tr] at 182 64
\pinlabel $b_6$ [t] at 122 93
\pinlabel $c_1$ [r] at 14 87
\pinlabel $c_2$ [l] at 38 43
\pinlabel $c_3$ [r] at 126 32
\pinlabel $c_4$ [l] at 194 64
\pinlabel $c_5$ [r] at 171 109
\pinlabel $c_6$ [l] at 82 122
\endlabellist
\begin{center}
\scalebox{.85}{\includegraphics{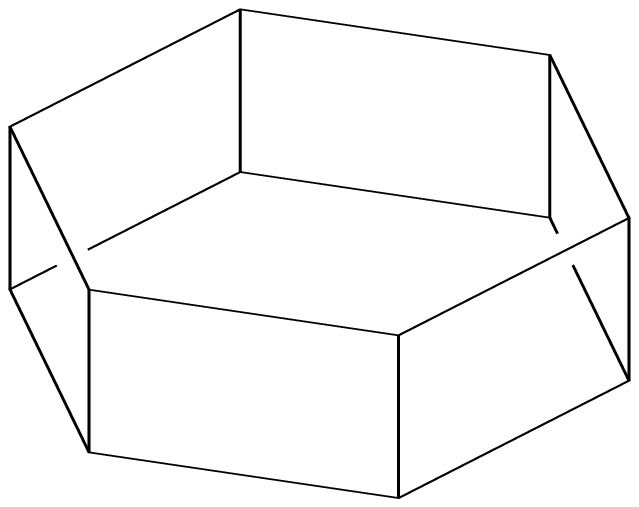}}
\end{center}
\caption{The $1$--skeleton of a hyperbolic $6$--prism}
\label{prism}
\end{figure}

Define $D_n$ to be the abstract $n$--prism.  By Andreev's theorem, the space
$\Omega(D_n)$ is naturally
parameterized as a convex polytope in $\RR^{3n}$ with coordinates given by
dihedral angles.  Define $O(D_n) \subset \Omega(D_n)$ to be the set of
labelings of the abstract
$n$--prism for which all dihedral angles are of the form $\pi/p$ for $p \in \ZZ$.
The elements of $\O(D_n)$ are realized by polyhedra that give discrete reflection
groups, so correspond to hyperbolic $3$--orbifolds and will be referred to as
\textit{Coxeter $n$--prisms}.

\subsection{Basic prisms}

Coxeter prisms were described completely by Derevnin and Kim in Theorem 5 of
\cite{derkim}.  The following lemma was discovered independently.

\begin{lemma}\label{redrep}
For any prism $\DD \in O(D_n)$, $n\geq 5$, there exists an angle--nondecreasing
deformation through prisms $\DD_t \in \Omega(D_n)$ with dihedral angles
$\aa_i(t),\,
\bb_i(t),$ and $\cc_i(t)$ from $\DD=\DD_0$ to  $\DD_1\in O(D_n)$ with
dihedral angles satisfying the following
properties, up to cyclic permutation of the indices:
\begin{enumerate}
	\item $\cc_1(1) = \cc_2(1) = \dots = \cc_n(1) = \pi/2$
	\item $\aa_1(1)=\bb_1(1)=\aa_2(1)=\bb_2(1)=\pi/2$
	\item For each $i$, $3 \leq i \leq n$, $(\aa_i(1),\bb_i(1))=(\pi/2,
	\pi/3)$ or $(\aa_i(1),\bb_i(1))=(\pi/3, \pi/2).$
\end{enumerate}
Furthermore, $\vol(\DD_0) \geq \vol(\DD_1)$.
\end{lemma}

\begin{proof}
Let $\DD \in O(D_n)$, $n>4$ with dihedral angles $\aa_i, \bb_i$, and $\cc_i$ as
above.  There are no prismatic $3$--circuits in $\DD$, so the only restrictions placed on $\DD$ by Andreev's
theorem, are that link of each vertex of $\DD$ is either a  Euclidean or
spherical triangle and for each pair $(i,j)$ with $1 \leq i\neq j \leq n$ with
$i\neq j\pm 1$, modulo $n$,
$\aa_i + \aa_j + \bb_i + \bb_j < 2 \pi$.  Condition (1) of the lemma follows
since increasing the $\cc_i$ to $\pi/2$ will increase the angle sum of the
link of each vertex, so will satisfy Andreev's theorem throughout the
deformation.   

By the fifth condition of Andreev's theorem, there are at most two
pairs $(\aa_i,\bb_i)=(\pi/2,\pi/2)$.  If there are two such pairs, they must be
adjacent, so without loss of generality, we may assume
$(\aa_1,\bb_1)=(\aa_2,\bb_2)=(\pi/2,\pi/2).$  If there is only one such pair, we
may assume that it is $(\aa_1,\bb_1)$.  Then, $(\aa_2,\bb_2)$ may be deformed to
$(\pi/2,\pi/2)$.  If there are no such pairs, then $(\aa_1,\bb_1)$ and
$(\aa_2,\bb_2)$ may be deformed to $(\pi/2,\pi/2)$.  This gives condition
(2).

After completing the deformations to satisfy (1) and (2), for each $i = 3, \dots
n$, at most $1$ of $\aa_i$ and $\bb_i$ is $\pi/2$.  If one of $\aa_i$ or $\bb_i$
is $\pi/2$, then the other is less than or equal to $\pi/3$, so may be increased
to $\pi/3$.  If neither $\aa_i$ nor $\bb_i$ is $\pi/2$, then the pair can
be increased to $(\pi/2, \pi/3)$.  This yields $\DD_1$ as described.

The deformations are all angle--increasing, so by Schl\"{a}fli's formula,
$\vol(\DD_0) \geq \vol(\DD_1).$ 
\end{proof}

An $n$--prism that satisfies the conclusion of Lemma~\ref{redrep} will be
referred to as a \textit{basic $n$--prism}. Define an \textit{alternating
$n$--prism} to be a basic $n$--prism  for which $\aa_3=\bb_4=\aa_5= \dots = \aa_n$
or $\bb_n$ if $n$ is odd or even, respectively.  Note that for each $n$,
there is only one alternating $n$--prism in $O(D_n)$ up to isometry.  It will
be shown in Section~\ref{S:altmin} that the alternating $n$--prism is the
Coxeter $n$--prism of smallest volume.

\subsection{Cubes} \label{S:cubes}  In this section, we analyze the geometry
of two types of $4$--prisms into which any basic prism may be decomposed.  We
prove three technical lemmas that will be used to identify the $n$--prism of
minimal volume.

For $\mu \in [0,\pi/2)$, define $C_1(\mu)$ to be the $4$--prism with
$\aa_3=\bb_4=\pi/3$, $\cc_1=\mu$, and all other dihedral angles equal to
$\pi/2$.  A cube such as $C_1$ where all dihedral angles are $\pi/2$ except for
$\aa_3$, $\bb_4$, and $\cc_1$ is known as a \textit{Lambert cube.}  The angles
$\aa_3$, $\bb_4$, and $\cc_1$ are the \textit{essential angles} of the Lambert
cube.  Define
$C_2(\mu)$ to be the $4$--prism with $\aa_3=\aa_4=\pi/3$, $\cc_1=\mu$, and all
other dihedral angles equal to $\pi/2$.   For $i=1,2$, let
$\rho_i(\mu)$ be the hyperbolic length of the edge  having dihedral angle $\mu$.  See
Figure~\ref{fig:cubes}.  Define $V_i(\mu) = \vol(C_i(\mu))$.

\begin{figure}
\labellist
\small\hair 2 pt
\pinlabel $\frac{\pi}{3}$ [b] at 95 62
\pinlabel $\frac{\pi}{3}$ [tl] at 149 25
\pinlabel $\frac{\pi}{3}$ [b] at 328 62
\pinlabel $\frac{\pi}{3}$ [br] at 378 83
\pinlabel $\rho_1(\mu)$ [r] at 3 95
\pinlabel $\rho_2(\mu)$ [r] at 235 95
\pinlabel $\mu$ [l] at 26 93
\pinlabel $\mu$ [l] at 258 93
\pinlabel $C_1(\mu)$ [t] at 93 -10
\pinlabel $C_2(\mu)$ [t] at 331 -10
\endlabellist
\begin{center}
\scalebox{.85}{\includegraphics{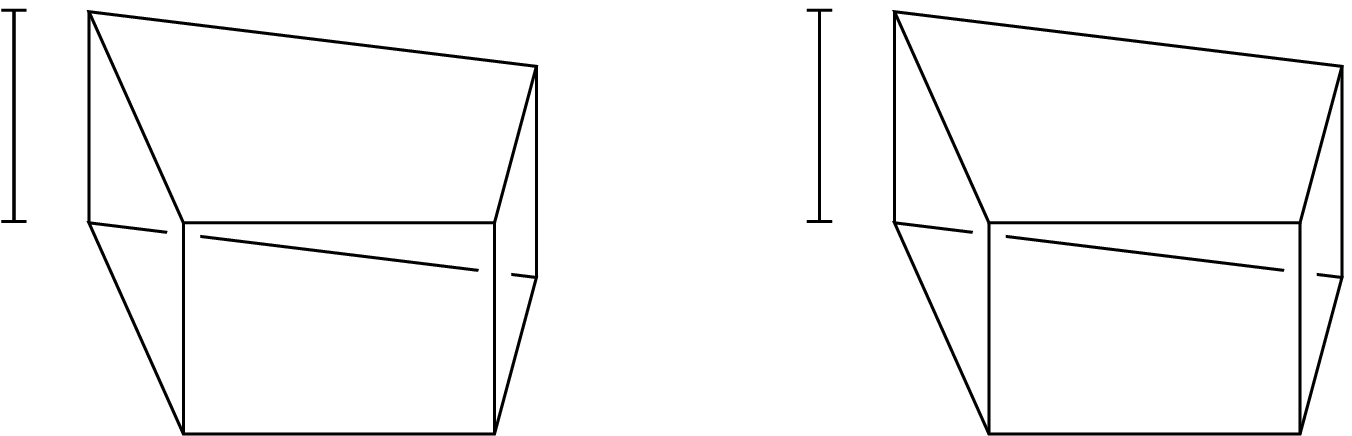}}
\end{center}
\caption{$C_1(\mu)$ and $C_2(\mu)$}
\label{fig:cubes}
\end{figure}

The following lemma shows that $\rho_i(\mu)$ is determined by $\mu$.

\begin{lemma}\label{gram}
Let $\mu \in [0,\pi/2)$.  Then, 
$$\cosh(\rho_1(\mu)) = 
\sqrt{\frac{1+ 24 \cos^2{\mu} + \sqrt{1+ 48 \cos^2{\mu}}}{32\cos^2{\mu}}},$$
and
$$\cosh(\rho_2(\mu)) = 
\sqrt{\frac{3\cos{\mu} + 1}{4\cos{\mu}}}.$$
\end{lemma}

\begin{proof}
We will work in the Lobachevsky model of $\HH^3$ in this proof.
Consider the Gram matrix $G(C_1(\mu))$ for $C_1(\mu)$.  Recall that for a
polyhedron $\PP$,
$$G(\PP) = \left[ w_i \cdot w_j \right],$$ 
where $i,\, j \in \{1,\, \dots, \, |\text{Faces}(\PP)|\}$, 
$w_i$ is the outward unit normal vector to the face $F_i$ of $\PP$ and the inner product is
defined by
\begin{equation}
	w_i \cdot w_j=
		\begin{cases}
		1		  & \text{if } i=j, \\
		-\cos{\theta_{ij}} & \text{if }F_i \text{ and } F_j \text{ meet
			with dihedral angle }\theta_{ij},\text{ or} \\
		-\cosh{d(F_i,F_k)} & \text{if } F_i \cap F_j = \emptyset.
		\end{cases}
\end{equation}

Let $F_1$ be the face bounded by the edges $a_i$, $i=1,\,\dots,\,4$, $F_2$ the face
bounded by $b_i$, $i=1,\, \dots,\, 4$, $F_3$ and $F_4$ the faces containing the
edge $c_1$, and $F_5$ and $F_6$ the remaining faces chosen so that $F_3 \cap F_5
= \emptyset$ and $F_4 \cap F_6 = \emptyset$.  Set $x_1 =
-\cosh{d(F_1,F_2)},$ $x_2=-\cosh{d(F_3,F_5)}$, $x_3 = -\cosh{d(F_4,F_6)},$ and
$m=-\cos{\mu}$.  Note that $x_1 = -\cosh(\rho_1(\mu))$.  The Gram matrix for
$C_1(\mu)$ then is given by 
$$G(C_1(\mu)) = 
\left[	
\begin{array}{cccccc}
1 & x_1 & 0 & -\frac{1}{2} & 0 & 0 \\
x_1 & 1 & 0 & 0 & -\frac{1}{2} & 0 \\
0 & 0 & 1 & 0 & x_2 & m \\
-\frac{1}{2} & 0 & 0 & 1 & 0 & x_3 \\
0 & -\frac{1}{2} & x_2 & 0 & 1 & 0 \\
0 & 0 & m & x_3 & 0 & 1  
\end{array}
\right].$$

Since the vectors $w_i$ are in $\RR^{3,1}$, this matrix can have at most rank
$4$.  Hence, deleting any row and column leaves a $5$ by $5$ matrix with
determinant equal to $0$.  Deleting the second row and column gives the equation
$$x_2^2 x_3^2 -x_2^2- \frac{3}{4} x_3^2 + \frac{3}{4}(1 - m^2 ) = 0.$$  
Repeating for the fifth row and column and for the sixth row and column yields
$$x_{1}^2 x_3^2 + (m^2-1)x_{1}^2 -x_3^2+ \frac{3}{4}(1 - m^2 ) = 0$$
and
$$x_{1}^2 x_2^2 - x_{1}^2 - \frac{3}{4}x_2^2 + \frac{9}{16} = 0.$$

This system of equations is guaranteed a unique solution with $x_1,\, x_2,\, x_3<0$ by
Andreev's theorem.  A computation shows that there is
such a solution with 
$$x_1 = -\sqrt{\frac{24 \cos^2{\mu} +1 + \sqrt{1+ 48 \cos^2{\mu}}}{32\cos^2{\mu}}},$$ which
completes the proof of the first half of the lemma.  The second claim of the
lemma is proved via identical methods.
\end{proof}

The next lemma exhibits the relationship between the volume of $C_1(\mu)$ and
$C_2(\mu)$.

\begin{lemma}\label{cubelemma}
For all $\mu \in [0,\pi/2)$, $V_1(\mu)<V_2(\mu)$.
\end{lemma}

\begin{proof}
For $i=1,2$, let $\Phi_i : [0,\mu] \to
O(D_4)$ be a smooth
deformation from $C_i(0)$ to $C_i(\mu)$ so that $\cc_1(t)=t$.  Then, by
Schl\"{a}fli's formula, $$V_i(\mu)-V_i(0) = -\frac{1}{2} \int_0^{\mu}
\rho_i(t)\, dt,$$
for $i=1,2$.
Subtracting $V_1(\mu)$ from $V_2(\mu)$ gives
$$V_2(\mu)-V_1(\mu) = (V_2(0) - V_1(0)) + \frac{1}{2} \int_{0}^{\mu}
(\rho_1(t)-\rho_2(t))\, dt.$$
To show that $V_1(\mu)<V_2(\mu)$, it suffices to show that $V_1(0) < V_2(0)$ and
that $\rho_2(t) \leq \rho_1(t)$ for all $t \in [0,\mu]$.

If $C$ is a $4$--prism with $0\leq \aa_3<\pi/2$, $0 \leq \bb_4 < \pi/2$,
$\cc_1=0$, and all other dihedral angles $\pi/2$, then a special case of a
result of Kellerhals in \cite{kellerhals} gives

$$\vol(C)=\frac{1}{4} \bigg( \Lambda(\aa_3 + \theta) - \Lambda(\aa_3 - \theta)
+\Lambda(\bb_4 + \theta) - \Lambda(\bb_4 - \theta) +4
\Lambda\left(\frac{\pi}{2} - \theta \right) \bigg),$$
where
$$0<\theta=\arctan{\frac{\sqrt{1-\sin^2{\aa_3}
\sin^2{\bb_4}}}{\cos{\aa_3}\cos{\bb_4}}},$$
and $\Lambda$ is the Lobachevsky function.  Recall that the Lobachevsky
function is defined by 
$$\Lambda(\theta) = - \int_0^{\theta} \log |2\sin t| \, dt.$$
Using this formula, $V_1(0)$ can be calculated and is approximately $.44446$.	

In the upper half--space model for $\HH^3$, $C_2(0)$ may be obtained by gluing
together two copies of the polyhedron  bounded by the planes $x=0$, $y=0$,
$x=\sqrt{2}/2$, $y = 1/2$ and bounded below by the unit hemisphere centered
at the origin. By explicitly calculating the integral of the hyperbolic volume
form over this polyhedron, it is seen that $V_2(0)$ is approximately $.50192$.  
Hence, $V_1(0)<V_2(0)$.

Using the formulas derived in Lemma~\ref{gram}, a direct calculation shows that
$\rho_1(t) \geq \rho_2(t)$ for all $t\in [0,\pi/2)$, which completes the proof.
\end{proof}

Finally, we show that the function $V_1$ is convex.

\begin{lemma}\label{convex}
The function $V_1(t) = \vol(C_1(t))$ is convex for $t \in [0,\pi/2]$.
\end{lemma}

\begin{proof}
By the Schl\"{a}fli differential formula, it suffices to show that
$$\frac{d}{dt} \rho_1(t) > 0$$
for $t\in[0,\pi/2]$ where
$$\cosh(\rho_1(t)) = 
\sqrt{\frac{24 \cos^2{t} +1 + \sqrt{1+ 48 \cos^2{t}}}{32\cos^2{t}}}.$$
One can check that $\frac{d}{dt}{\cosh^2(\rho_1(t)})>0$ for $t \in [0,\pi/2]$ by
differentiating.  Therefore, since
$\frac{d}{dx}{x^2}>0$ and $\frac{d}{dx}{\cosh(x)}>0$, it follows from the
chain rule that
$\frac{d}{dt}\rho_1(t)>0$ as well.  \end{proof}

\subsection{Alternating prism is minimal volume}\label{S:altmin}
In this section, we apply the preceding lemmas to show that the alternating
prism is of minimal volume and prove the lower bound in
Theorem~\ref{prismsum}.

This lemma describes a decomposition of basic prisms into cubes of the form
$C_1(\mu)$ and $C_2(\nu)$.  It is a special
case of Theorem 4 in \cite{derkim} and was independently discovered by the
author.

\begin{lemma}\label{ddecomp}
Suppose that $\DD$ is a  basic $n$--prism.  Then $\DD$ can be decomposed
into $r$ copies of $C_1(\mu)$ and $s=n-r-3$ copies of $C_2(\nu)$ where
$r\mu+s\nu=\pi/2$.
\end{lemma}

\begin{proof}
Label the quadrilateral face bounded by $a_i, b_i, c_i,$ and $c_{i-1}$ by
$F_i$.  For each $F_i$, $4\leq i \leq n-1,$  there is a unique geodesic plane
that contains $c_1$ and meets $F_i$ orthogonally.  Decomposing along these
planes gives the desired decomposition into copies of $C_1$ and $C_2$.  The fact
that the determining angles of each copy of $C_i$, $i=1,2$, are equal follows
from the fact that the length of $c_1$ determines the dihedral angle $\mu$ by 
Lemma~\ref{gram}.
\end{proof}

Note that Lemma~\ref{ddecomp} gives a decomposition of the alternating
$n$--prism into $n-3$ copies of $C_1\left(\frac{\pi}{2(n-3)}\right)$.

\begin{theorem}  \label{altmin}
The alternating $n$--prism is the minimal volume prism in $O(D_n)$.
\end{theorem}
\begin{proof}
If $\DD \in O(D_n)$ is not a basic prism, then by Lemma \ref{redrep}, there is a basic
prism with volume smaller than $\DD$.  Therefore the it suffices to show that the
alternating $n$--prism is the smallest volume basic $n$--prism.

By Lemma~\ref{ddecomp}, it is enough to show that 
$$rV_1(\mu)+sV_2(\nu) > (r+s)V_1\left(\frac{\pi}{2(r+s)}\right),$$
where $r\mu + s \nu = \pi/2$.  Setting $t=\frac{r}{r+s}$, the inequality becomes
$$ t V_1(\mu) + (1-t) V_2(\nu) > V_1 ( t\mu + (1-t) \nu).$$
This inequality follows immediately from the fact that $V_2 > V_1$ and the
convexity of $V_1$.
\end{proof}

Lemma~\ref{ddecomp} can be used to express the volume of any basic prism in terms of the
volume of $C_1$ and $C_2$.  In particular, the volume of the alternating prisms
can be calculated explicitly:

\begin{corollary}\label{volalt}
The volume of the alternating $n$--prism $A_n$ is given by
$$\vol(A_n)=(n-3)\vol\left(C_1\left(\frac{\pi}{2(n-3)}\right)\right).$$
\end{corollary}

The quantity $\vol(C_1(\frac{\pi}{2(n-3)}))$ can be calculated using a theorem
of Kellerhals that we restate here \cite{kellerhals}.  Suppose that $C$ is a Lambert cube with essential angles
$\aa_3,$ $\bb_4$, and $\cc_1$. The \textit{principal parameter}, $\theta$, of
$C$ is defined by $$\theta=\arctan{\frac{\sqrt{\cosh^2{\rho(\cc_1)} -
\sin^2{\aa_3} \sin^2{\bb_4}}}{\cos{\aa_3}\cos{\bb_4}}},$$ where $\rho(\cc_1)$ is
the length of the edge $c_1$.  The volume of the Lambert cube is then given by
the following theorem.

\begin{theorem}[Kellerhals]\label{kellerhals}
Let $C$ be a Lambert cube with essential angles $0 \leq \aa_3,\, \bb_4,\, \cc_1
\leq \pi/2$.  Then the volume of $C$ is given by 
\begin{align*}
\vol(C) = \frac{1}{4} \bigg(&\Lambda(\aa_3 + \theta) - \Lambda(\aa_3 - \theta) +
\Lambda(\bb_4 + \theta) - \Lambda(\bb_4 - \theta) \\ 
&+ \Lambda(\cc_1 + \theta) - \Lambda(\cc_1 - \theta) - \Lambda(2\theta) +
2\Lambda\left(\frac{\pi}{2} - \theta\right)\bigg).
\end{align*}
\end{theorem}

Corollary~\ref{volalt} only needs the case where $\aa_3=\bb_4=\pi/3$.
To find the principal parameter, Lemma~\ref{gram} can be used to compute that
$$\cosh{\rho(\cc_1)}= 
\sqrt{\frac{1+ 24 \cos^2{\cc_1}  + \sqrt{1+ 48
\cos^2{\cc_1}}}{32\cos^2{\cc_1}}}.$$
A program such as Mathematica easily computes the volume of Lambert cubes using
Kellerhals' formula.

Finally, it should be noted that for all $n \geq 5$, by Schl\"{a}fli's formula 
$$\vol\left(C_1\left(\frac{\pi}{2(n-3)}\right)\right) > \vol\left( C_1 \left(
\frac{\pi}{3}\right) \right) \approx .324423,$$
so we have the following corollary to Corollary~\ref{volalt} that bounds the
volume of the $n$--prism from below linearly in $n$.  This proves the lower
bound in Theorem~\ref{prismsum}.

\begin{corollary}\label{corprism}
For any Coxeter $n$--prism $D$,
$$\vol(D) > (n-3) \cdot \vol\left( C_1 \left(
\frac{\pi}{3}\right) \right).$$
\end{corollary}

\subsection{Prism regions in non--obtuse polyhedra}\label{gutsofprisms}

For a turnover--reduced non--obtuse polyhedron $\PP$, Theorem~\ref{bspolyorb}
applied to $\QQ_{\PP}^{\perp}$ gives a collection, $\TT$,  of topological
quadrilaterals along which $\PP$ may be decomposed into atoroidal components
and prisms.  We have already shown how to bound below the volume of the
atoroidal components.  In this section, we will show how to obtain a lower
volume bound on the components of the complement of $\TT$ in $\PP$ that
correspond to the Seifert--fibered components in the splitting of
$\QQ_{\PP}^{\perp}$.  We may assume that $\PP$ has all dihedral angles equal
to $\pi/2$ or $\pi/3$ because by Proposition~\ref{unif}, there exists a
volume--nonincreasing deformation from any turnover--reduced Coxeter
polyhedron to one with all dihedral angles $\pi/2$ or $\pi/3$.

Let $\TT' \subset \TT$ consist of the quadrilaterals in $\TT$ that meet both a
prism and an atoroidal component of the complement of $\TT$.  Denote by $\PP
\split \TT'$ the disjoint union of the closures of the components of $\PP
\setminus \TT'$.  Each component of $\PP \split \TT'$ is either a collection of
atoroidal components glued along $\TT \setminus \TT'$ or a collection of prisms
glued along $\TT \setminus \TT'$.  Denote the components of $\PP \split \TT'$
that consist of prisms glued to one another by $\GG_1 , \GG_2, \dots \GG_N$.
For each $\GG_i$, the associated orbifold $\QQ_{\GG_i}^{\perp}$ is a graph
orbifold.   The \textit{boundary} of $\GG_i$, denoted $\partial{G_i}$, is
$\GG_i \cap \TT'$.  Each edge of $\GG_i$ that is intersected by a
quadrilateral in $\TT \setminus \TT'$ is a \textit{shared edge}, in that it is
an edge of two prisms that have been glued together.  Note that, in general,
$\GG_i$ is not a hyperbolic polyhedron because the boundary of $\GG_i$ will
only be geodesic in special cases.  See Figure~\ref{graphprism} for an example
of a possible $\GG_i$.

\begin{figure} 
\labellist
\small\hair 2pt
\pinlabel $e_1$ [bl] at 261 471
\pinlabel $e_2$ [tr] at 290 495
\pinlabel $e_3$ [tl] at 290 428
\pinlabel $e_4$ [tl] at 260 397
\endlabellist
\begin{center}
\scalebox{1}{\includegraphics{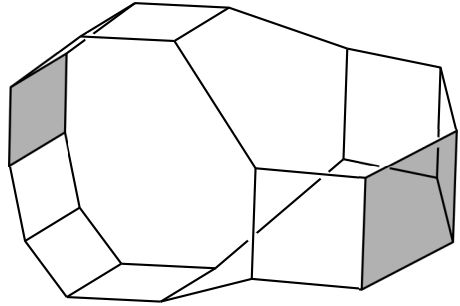}} 
\end{center}
\caption{A possible $\GG_i$.  The shaded faces comprise $\partial{\GG_i}$.  The
edges $e_1,$ $e_2$, $e_3$, and $e_4$ are shared edges} 
\label{graphprism} 
\end{figure}

To bound the volume from below, we first decompose each $\GG_i$ further into its
constituent prisms by considering the components $\DD_i^j$ of $\GG_i \split
(\TT \setminus \TT')$. Again, the $\DD_i^j$ are not hyperbolic polyhedra, in
general.  We label the edges as before, whether or not they are actual
geodesic edges as in Figure~\ref{prismwithbdy}.  Label the face bounded by
$a_i, b_i, c_i,$ and $c_{i-1}$ by $F_i$.   The non--geodesic edges will be
referred to as \textit{virtual edges}.  The \textit{order} of $\DD_i^j$ is
$n$ if $\DD_i^j$ has $n+2$ total faces, including the non--geodesic faces.

Recall that Lemma~\ref{redrep}, which follows from Andreev's theorem, says
that there exists a volume--nonincreasing deformation from any Coxeter prism
to one where all dihedral angles are $\pi/2$ or $\pi/3$ such that all of the
edges $c_l$ have dihedral angle $\pi/2$, two adjacent quadrilateral faces have all
dihedral angles $\pi/2$ and each other quadrilateral face has
exactly one dihedral angle of $\pi/3$.  A similar statement is true for the
$\DD_i^j$.  

\begin{figure} 
\labellist
\small\hair 2pt
\pinlabel $a_1$ [br] at 22 84
\pinlabel $a_2$ [bl] at 8 56
\pinlabel $a_3$ [b] at 47 37 
\pinlabel $a_4$ [br] at 97 43
\pinlabel $a_5$ [l] at 116 71
\pinlabel $a_6$ [b] at 77 93
\pinlabel $b_1$ [tl] at 27 56
\pinlabel $b_2$ [tr] at 7 29
\pinlabel $b_3$ [t] at 45 5
\pinlabel $b_4$ [tl] at 101 11
\pinlabel $b_5$ [tr] at 116 43
\pinlabel $b_6$ [t] at 74 62
\pinlabel $c_1$ [r] at 0 56
\pinlabel $c_2$ [l] at 16 25
\pinlabel $c_3$ [r] at 77 17
\pinlabel $c_4$ [l] at 122 40
\pinlabel $c_5$ [r] at 107 71
\pinlabel $c_6$ [l] at 47 81
\pinlabel $\frac{a_1}{b_1}$ [br] at 191 85
\pinlabel $\frac{a_2}{b_2}$ [r] at 171 48
\pinlabel $\frac{a_3}{b_3}$ [tr] at 191 11
\pinlabel $\frac{a_4}{b_4}$ [tl] at 237 14
\pinlabel $\frac{a_5}{b_5}$ [l] at 254 47
\pinlabel $\frac{a_6}{b_6}$ [bl] at 233 85
\endlabellist
\begin{center}
\scalebox{1}{\includegraphics{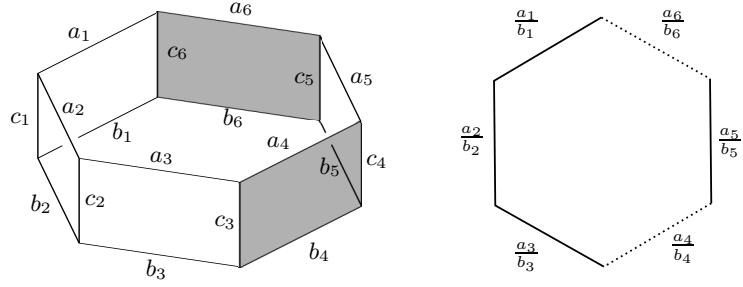}} 
\end{center}
\caption{Edge labels for a component $\DD_i^j$ with boundary faces shaded.  The
figure on the right is a useful schematic for representing such a prism region}
\label{prismwithbdy} 
\end{figure}

\begin{lemma}\label{redrepfordrs}
Suppose $\PP$ is a hyperbolic Coxeter polyhedron.  Then there exists a
volume--nonincreasing deformation of $\PP$ through hyperbolic polyhedra so that each
$\DD_i^j$ in the decomposition described above of the resulting polyhedron has
the dihedral angles satisfying the following conditions, up to cyclic relabeling: 
\begin{enumerate}
	\item $\aa_1=\bb_1=\pi/2$
	\item $\cc_k = \pi/2$ if $c_k$ is not a virtual edge.
	\item If $a_2$ and $b_2$ are not virtual edges, then either
$(\aa_2, \bb_2) = (\pi/2, \pi/2)$, $(\aa_2, \bb_2) = (\pi/3, \pi/2)$,
$(\aa_2, \bb_2) = (\pi/2, \pi/3),$ or $(\aa_2, \bb_2) = (\pi/3, \pi/3).$
	\item For each $k$, $3 \leq k \leq n$, such that $a_k$ and $b_k$ are not
virtual edges, $(\aa_k,\bb_k)=(\pi/2,
	\pi/3)$,  $(\aa_k,\bb_k)=(\pi/3, \pi/2),$ or $(\aa_2, \bb_2) = (\pi/3,
\pi/3).$
\end{enumerate}
\end{lemma}

\begin{proof}
The proof here is essentially the same as the proof of Lemma~\ref{redrep}.  The
first two conditions preclude the existence of any prismatic $4$--circuits
passing through all edges with dihedral angles $\pi/2$.  By
Proposition~\ref{unif}, it is certainly the case that all other dihedral angles
in each prism region that are less than $\pi/3$ can be deformed to be $\pi/3$.
After this deformation, any pair of edges, $(a_k, b_k)$, that are not shared
edges with $\aa_k = \bb_k = \pi/3$ can have either
$\aa_k$ or $\bb_k$ deformed to $\pi/2$.  Figure~\ref{33eg} shows an example
where the dihedral angles along
a shared edge pair must both remain $\pi/3$.  
\end{proof}

\begin{figure}
\labellist
\hair 2 pt
\pinlabel $\frac{\pi}{2}$ [r] at 0 75
\pinlabel $\frac{\pi}{2}$ [r] at 21 124
\pinlabel $\frac{\pi}{2}$ [l] at 258 68
\pinlabel $\frac{\pi}{2}$ [l] at 278 117
\pinlabel $\frac{\pi}{3}$ [t] at 152 157
\pinlabel $\frac{\pi}{3}$ [t] at 126 103
\pinlabel $\frac{\pi}{2}$ [tl] at 178 19
\pinlabel $\frac{\pi}{2}$ [tl] at 217 89
\endlabellist
\begin{center}
\scalebox{.7}{\includegraphics{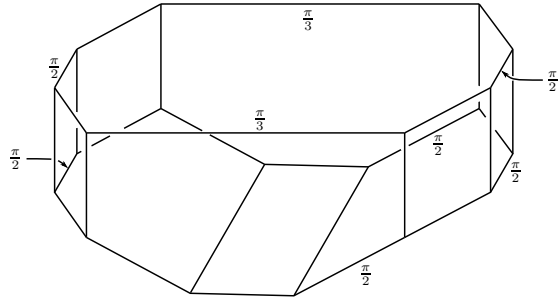}}
\end{center}
\caption{An example where a shared edge pair must both have dihedral angle of
$\pi/3$}
\label{33eg}
\end{figure}

From now on we will assume that $\PP$ satisfies the
conclusion of Lemma~\ref{redrepfordrs}.  In what follows, we give a
decomposition of the prism regions contained in $\PP$ and show how this
decomposition leads to a lower bound on the volume.  The
decompositions that follow should be thought of taking place in $\PP$ with the
previous decomposition into prisms used only as a mental crutch to understand
different ``regions" within $\PP$.

A (topological) quadrilateral $T$ embedded in $\PP$ is \textit{cylindrical} if
there exists a prismatic $4$--circuit  in $\PP \split T$ that  intersects two
edges of the boundary component of $\PP \split T$ corresponding to $T$ and two
edges with dihedral angle $\pi/2$.  See Figure~\ref{cylindrical}.  A quadrilateral is
\textit{acylindrical} if it is not cylindrical.

\begin{figure}
\labellist
\small\hair 2 pt
\pinlabel $T$ [bl] at  151 75
\pinlabel $\frac{\pi}{2}$ [br] at  58 128
\pinlabel $\frac{\pi}{2}$ [br] at  58 46
\endlabellist
\begin{center}
\scalebox{1}{\includegraphics{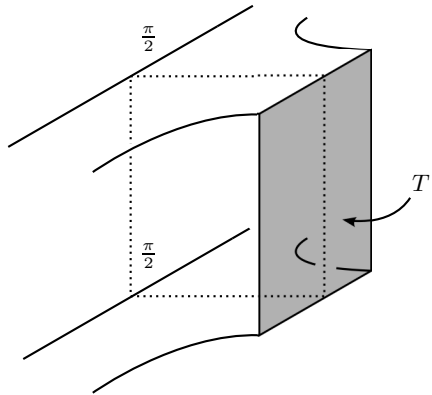}}
\end{center}
\caption{The shaded quadrilateral $T$ is cylindrical}
\label{cylindrical}
\end{figure}

\begin{lemma}\label{splitacyl}
If $T$ is an acylindrical quadrilateral in a hyperbolic Coxeter polyhedron
$\PP$, then each component, $\PP_i$, $i\in \{ 1,2 \}$, of $\PP \split T$ admits
a structure as a  hyperbolic polyhedron with dihedral angles along the edges
of  $T\cap \PP_i$ equal to $\pi/2$.  Furthermore, $$\vol(\PP) \geq
\vol(\PP_1) + \vol( \PP_2).$$
\end{lemma}

To prove this lemma, we again apply the theorem of Agol, Storm, and W. Thurston
\cite{ast}.

\begin{proof}
For the first claim, it suffices to show that each $\PP_i$ satisfies the
conditions of Andreev's theorem (Theorem~\ref{and})  when each of the edges
of $T \cap \PP_i$ are given dihedral angle $\pi/2$. The argument to show this
is the same as the argument used to prove Proposition~\ref{topdecomp}.

To prove the second claim, we use Theorem~\ref{ast}.  By Selberg's Lemma,
there exists a $n$--sheeted regular cover $M$ of $\HH^3 / \Gamma(\PP)$ that
is a hyperbolic $3$--manifold \cite{selberg}.  The acylindrical quadrilateral,
$T$, lifts to a orientable, incompressible surface $\Sigma$ embedded in $M$.
The components, $M_1$ and $M_2$, of $M \split \Sigma$ are index $n$ covers of
$\HH^3 / \Gamma(\PP_1)$ and $\HH^3 / \Gamma(\PP_2)$ with covering maps
induced by the covering map of $M$ to $\PP$.  Being finite regular covers of
hyperbolic orbifolds with geodesic boundary, each of the $M_i$ are hyperbolic
manifolds with geodesic boundary.  Hence, $\vol(M) = n \vol(\PP)$ and
$\vol(M_i) = n
\vol(\PP_i)$ for $i=1,2$.  Then, by Theorem~\ref{ast} and the fact that
$\frac{1}{2} V_3 ||D(M_i)|| =
\vol(M_i)$,  
\begin{align*}
n\vol(\PP) &= \vol(M) \geq \frac{1}{2} V_3 || D({M} \setminus\!\!\setminus \Sigma) || 
= \frac{1}{2} V_3 \left( || D(M_1) || + || D(M_2) || \right) \\
&=\vol(M_1) + \vol(M_2) = n\vol(\PP_1) + n\vol(\PP_2).
\end{align*}
\end{proof}

The following lemma is the basis for the decomposition of the graph orbifold
regions that will lead to to the lower bound.

\begin{lemma}\label{decomquad}
Let $\DD_i^j$ be a prism region with degree $n \geq 5$.  
\begin{enumerate}
\item If $\aa_1 = \bb_1 = \aa_2 = \bb_2 = \pi/2$, then for each $k$, $4 \leq k
\leq n-1$, such that $a_k$ and $b_k$ are not virtual edges, there exists a
geodesic quadrilateral containing $c_1$ and intersecting the face $F_k$
orthogonally.
\item If $\aa_1 = \bb_1 = \pi/2$ and $\aa_2 = \pi/3$ or $\bb_2 = \pi/3$,
then for each $k$, $3 \leq k \leq n-1$, there exists an acylindrical
quadrilateral that intersects the faces $F_1$ and $F_k$.  
\end{enumerate}
\end{lemma}

\begin{proof}
\begin{enumerate}
\item  For each $k \in \{ 4,\, \dots, \, n-1 \}$ such that $a_k$ and $b_k$
are not virtual edges, let $\Pi_k$ be the defining plane of the face of
$\DD_i^j$ containing $a_k$ and $b_k$.  Also, let $g_1 = \Pi_1 \cap \Pi_2$ be
the geodesic in which $c_1$ is contained.  It follows from the fact that
$\Pi_1$ and $\Pi_2$ are disjoint from $\Pi_k$ that
$g_1$ and $\Pi_k$ are disjoint. See, for example, Lemma 4.4 of
\cite{volpoly}. Hence there exists a geodesic plane, $\Pi,$
that contains $g_1$ and intersects $\Pi_k$ orthogonally.  That $\Pi$
intersects $\DD_i^j$ in a quadrilateral orthogonal to $F_k$ is a consequence
of the fact that all dihedral angles in $\PP$ are no more than $\pi/2$.  For
the case where $k = 4$ or $n-1$ and $a_k$ and $b_k$ are not virtual, the
quadrilaterals coincide with the faces $F_1$ and $F_2$.  
\item  Let $k \in
\{3,\, \dots, \, n-1 \}$.  Let $T_k$ be a quadrilateral that meets $F_1$ and
$F_k$.  If $T_k$ were cylindrical, the prismatic $4$--circuit realizing the
cylindricity, $\sigma$ must pass through two virtual edges contained of
$\DD_i^j$ because for all other pairs of non--virtual edges, $a_l$, $b_l,$ at
least one of the dihedral angles is $\pi/3$.  Suppose that the two edges of
$\sigma$ not in $T_k$ are $a$ and $b$.  Then there is a prismatic
$4$--circuit, $\sigma'$, passing through the edges $a$, $b$, $a_1$ and $b_1$,
all of which have dihedral angle $\pi/2$.  See Figure~\ref{extracir}.  This
contradicts Andreev's theorem, so $T_k$ must actually be acylindrical.
\begin{figure}
\labellist
\small\hair 2 pt
\pinlabel $T_k$ [t] at 199 0
\pinlabel $\partial{\DD_i^j}$ [t] at 105 31
\pinlabel $\sigma'$ [r] at 0 120
\pinlabel $\sigma$ [r] at 6 129
\pinlabel $a$ [tl] at 8 96
\pinlabel $b$ [tl] at 8 50
\pinlabel $a_1$ [t] at 232 73
\pinlabel $b_1$ [t] at 232 27
\endlabellist
\begin{center}
\scalebox{.85}{\includegraphics{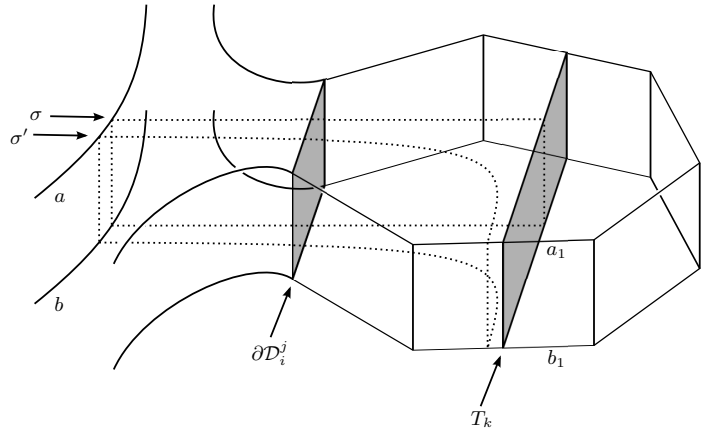}}
\end{center}
\caption{This illustrates part 2 of the proof of Lemma~\ref{decomquad}}
\label{extracir}
\end{figure}
\end{enumerate}
\end{proof}

We can now prove the lower bound on the volume of a prism region of $\PP$.

\begin{theorem}\label{mainthm}
Suppose that $\DD_i^j$ is a prism region of $\PP$ that contains $V \geq 2$
vertices of $\PP$.  Then, except for in the cases shown in Figure~\ref{badguys},
\begin{equation}
	\vol(\DD_i^j) >
		\begin{cases}
		 \left( \frac{V}{2} -3 \right) \cdot \vol\left( C_1 \left(
			\lambda \right)\right),& \text{if }V\geq 8\\
		C_1(\lambda), & \text{if }V = 2,\,4 \text{ or } 6,
		\end{cases}
\end{equation}
where $\lambda \in (0 , \pi/2)$ depends on $\DD_i^j$.
Moreover, if $V \geq 10$, then 
$$\vol(\DD_i^j) > \left( \frac{V}{2} -3 \right) \cdot \vol\left( C_1 \left(
\frac{\pi}{3} \right)\right), $$
where the value of $\vol\left( C_1 \left( \frac{\pi}{3} \right)\right)$ is
approximately $.324423$.
\end{theorem}

\begin{figure}
\labellist
\small\hair 2 pt
\pinlabel $\dots$ at 218 92
\pinlabel $\frac{2}{2}$ [tr] at 164 64
\pinlabel $\frac{2}{2}$ [t] at 218 32
\pinlabel $\frac{2}{2}$ [br] at 340 103
\pinlabel $\frac{2}{2}$ [bl] at 404 103
\pinlabel $\rotatebox{-78}{\dots}$ at 11 43
\pinlabel $\rotatebox{78}{\dots}$ at 115 43
\pinlabel $\frac{2}{2}$ [bl] at 96 103
\pinlabel $\dots$ at 374 0
\endlabellist
\begin{center}
\scalebox{.78}{\includegraphics{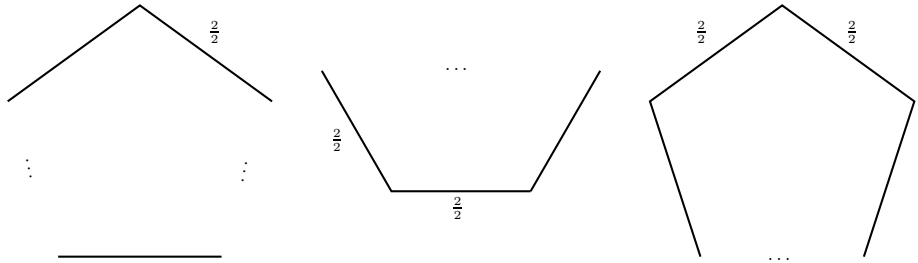}}
\end{center}
\caption{Theorem~\ref{mainthm} does not provide a lower bound on the volume in
these examples.  The ellipses may be replaced with any odd number of edges
that
alternate between dotted and solid and begin with a dotted edge.  Recall that
these diagrams are explained in Figure~\ref{prismwithbdy}}
\label{badguys}
\end{figure}

\begin{proof}
In each case, we will give a lower bound on the number of cubes of the form
$C_1$ or $C_2$ into which $\DD_i^j$ can be decomposed.  The proof finishes in
each case by applying the convexity argument used to prove Theorem~\ref{altmin}.

First suppose that $V \geq 8$, $\aa_1 = \aa_2 = \bb_1 = \bb_2 = \pi/2$, and
that $\DD_i^j$ is of order $n$.  Lemma~\ref{decomquad} says that for each $k
\in \{ 4,\, \dots, \, n-1 \}$ such that $a_k$ and $b_k$ are not virtual
edges, there exists a geodesic quadrilateral $T_k$ that contains $c_1$ and
intersects $F_k$ orthogonally.  Since $\DD_i^j$ contains at least $8$
vertices of $\PP$, there are at least $V/2 -3$ values of $k$, $4 \leq k \leq
n-1,$ such that $a_k$ and $b_k$ are not virtual, and such that $a_{k-1}$ and
$b_{k-1}$ are not virtual or $a_{k+1}$ and $b_{k+1}$ are not virtual.  If
$a_k$ and $a_{k-1}$ are both not virtual, then just as in the proof of
Lemma~\ref{ddecomp}, there is a cube $C_1(\mu)$ or $C_2(\nu)$ formed by
$T_k$, $T_{k-1}$, $F_k$ and $F_{k-1}$, for some $\mu,\,\nu \in (0,\pi/2)$.  A
similar statement is true if it is $a_{k+1}$ that is not virtual.  This
procedure gives a cube for each $k \in K$ where $$K = \{ 3\leq k \leq n-1 \,
| \, \text{the endpoints of }c_k\text{ are vertices of } \PP \}.$$ For each
of the at least $V/2 -3$ values of $k \in \{4, \dots, n-1\}$ such that $a_k$
and $b_k$ are not virtual, and such that $a_{k-1}$ and $b_{k-1}$ are not
virtual or $a_{k+1}$ and $b_{k+1}$ are not virtual, the endpoints of
$c_{k-1}$ or $c_k$, respectively, are vertices of $\PP$.  Therefore $|K| \geq
V/2 -3$, which completes this case of the proof.

Now suppose that $V\geq 8$ and either $a_2$ and $b_2$ are virtual,
$(\aa_2,\bb_2) = (\pi/2, \pi/3)$, or $(\aa_2, \bb_2) = (\pi/3, \pi/2)$.
Choose any value of $k\in \{ 3, \, \dots, \, n-1\}$ such that $a_k$ and $b_k$
are not virtual edges and $a_{k-1}$ and $b_{k-1}$ or $a_{k+1}$ and $b_{k+1}$
are virtual edges.  Suppose for concreteness that $a_{k-1}$ and $b_{k-1}$ are
virtual edges.  Then, there exists an acylindrical topological quadrilateral,
$T_k$, that intersects $F_1$ and $F_k$ by Lemma~\ref{decomquad}.  The prism
$\DD_i^j$, as well as the entire polyhedron, $\PP,$ can be split along this
quadrilateral.  By Lemma~\ref{splitacyl}, each component of $\PP \split T_k$
has a hyperbolic structure with $\PP \cap T_k$ totally geodesic and such that
the sum of the volume of the components is no more than the volume of $\PP$.
The prism region $\DD_i^j$ splits into two prism regions, each of which have a
pair of adjacent faces with $\aa_i$ and $\bb_i$ equal to $\pi/2$.  The
decomposition described in the previous case can now be applied to each
component.  The two resulting components yield the fewest cubes when each of
the edges $c_{n-1}$, $c_n$, $c_1$ and $c_2$ are not virtual.  In this case,
$\DD_i^j$ decomposes into $V/2-2$ cubes.

We now consider the case where $V=2,\, 4$ or $6$.  The argument is a case--by--case
analysis of the possible vertex configurations.  We will identify a single
cube of the form $C_1(\mu)$ or $C_2(\mu)$ in each case.  The argument then
finishes by using the fact that $\vol(C_2(\mu)) > \vol(C_1(\mu))$.

Suppose that $V=2$.  There are two cases here.  First, suppose that
$\aa_l = \bb_l = \pi/2,$ where the labeling is as in Figure~\ref{2verts}.  In this
case, Lemma~\ref{decomquad} implies that there exists an acylindrical
quadrilateral $T_1$.  The component of $\DD_i^j \split T_1$ containing the vertices of
$\PP$ then has volume at least $C_1(\mu)$ for some $\mu \in (0, \pi/2)$ by
applying the argument from above.  In the other case where $\aa_1 = \bb_1 =
\pi/2$ or $\aa_1 = \bb_1 = \pi/2$, there is no acylindrical quadrilateral along which to decompose.  

\begin{figure}
\labellist
\small\hair 2 pt
\pinlabel $\rotatebox{-78}{\dots}$ at 11 43
\pinlabel $\rotatebox{78}{\dots}$ at 115 43
\pinlabel $\frac{a_1}{b_1}$ [bl] at 96 102
\pinlabel $\frac{a_2}{b_2}$ [br] at 32 102
\pinlabel $\frac{a_l}{b_l}$ [t] at 64 0
\endlabellist
\begin{center}
\scalebox{1}{\includegraphics{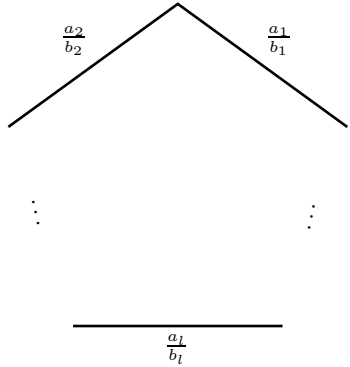}}
\end{center}
\caption{The $2$--vertex case}
\label{2verts}
\end{figure}

Next suppose that $V=4$.  There are two possible configurations of vertices
here.  Either the two pairs of vertices of $\PP$ are separated by boundary
components or they are not.  See Figure~\ref{4goodguys}.  The previous techniques
suffice to find a cube except for in the case where $\aa_1 = \bb_1 = \aa_2 =
\bb_2$ and the vertices are not separated by virtual edges as in the
middle diagram of Figure~\ref{badguys}, where there is no acylindrical quadrilateral.

\begin{figure}
\labellist
\small\hair 2 pt
\pinlabel $\dots$ at 64 95
\pinlabel $\rotatebox{90}{\dots}$ at 164 71 
\pinlabel $\rotatebox{90}{\dots}$ at 280 71 
\pinlabel $\frac{a_1}{b_1}$ [tr] at 9 71
\pinlabel $\frac{a_2}{b_2}$ [t] at 64 46
\pinlabel $\frac{a_3}{b_3}$ [tl] at 116 71 
\pinlabel $\frac{a_1}{b_1}$ [tr] at 192 17
\pinlabel $\frac{a_2}{b_2}$ [t] at 253 17
\pinlabel $\frac{a_l}{b_l}$ [br] at 192 119
\pinlabel $\frac{a_{l+1}}{b_{l+1}}$ [bl] at 253 119
\endlabellist
\begin{center}
\scalebox{.8}{\includegraphics{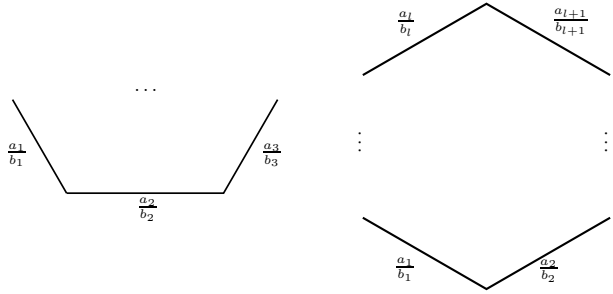}}
\end{center}
\caption{The $4$--vertex case}
\label{4goodguys}
\end{figure}

When $V=6$, there are three possible configurations of vertices.  Either none
are separated from any other by virtual edges, a single pair is isolated
or all three are mutually isolated.  See Figure~\ref{6cases}. Again the previous
methods produce at least one cube in all cases except for the case of the
leftmost diagram in Figure~\ref{6cases} where $\aa_2 = \bb_2 = \aa_3 = \bb_3 = \pi/2$
where there is no acylindrical quadrilateral.

\begin{figure}
\labellist
\small\hair 2 pt
\pinlabel $\dots$ at 63 1
\pinlabel \rotatebox{75}{$\dots$} at 166 79
\pinlabel \rotatebox{-75}{$\dots$} at 284 79
\pinlabel \rotatebox{60}{$\dots$} at 334 98
\pinlabel \rotatebox{-60}{$\dots$} at 442 98
\pinlabel $\dots$ at 387 1
\pinlabel $\frac{a_1}{b_1}$ [tr] at 10 40
\pinlabel $\frac{a_2}{b_2}$ [br] at 29 100
\pinlabel $\frac{a_3}{b_3}$ [bl] at 93 104
\pinlabel $\frac{a_1}{b_1}$ [tr] at 177 25
\pinlabel $\frac{a_2}{b_2}$ [t] at 225 1
\pinlabel $\frac{a_3}{b_3}$ [tl] at 277 26
\pinlabel $\frac{a_4}{b_4}$ [tl] at 117 39
\pinlabel $\frac{a_l}{b_l}$ [bl] at 252 119
\pinlabel $\frac{a_{l+1}}{b_{l+1}}$ [br] at 199 119
\pinlabel $c_l$ [b] at 226 130
\pinlabel $\frac{a_{1}}{b_{1}}$ [r] at 323 54
\pinlabel $\frac{a_{2}}{b_{2}}$ [tr] at 345 16
\pinlabel $\frac{a_l}{b_l}$ [tl] at 430 15
\pinlabel $\frac{a_{l+1}}{b_{l+1}}$ [l] at 452 53
\pinlabel $\frac{a_m}{b_m}$ [bl] at 411 125
\pinlabel $\frac{a_{m+1}}{b_{m+1}}$ [br] at 367 126
\pinlabel $c_1$ [r] at 0 77
\endlabellist
\begin{center}
\scalebox{.76}{\includegraphics{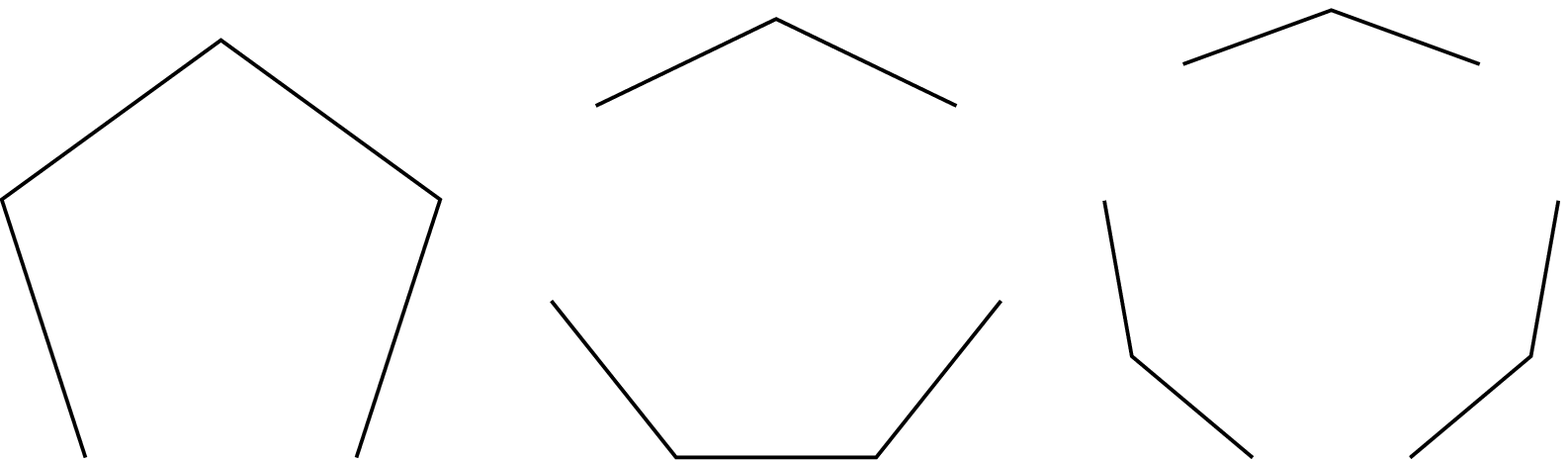}}
\end{center}
\caption{The $6$--vertex case.  The ellipses may be filled in with either a dotted segment or an
alternating sequence of dotted and solid segments of odd length beginning with a
dotted segment}
\label{6cases}
\end{figure}

The second statement follows from the same convexity argument and the fact that
for $V \geq 10$, there are at least $2$ cubes in the decomposition.
\end{proof}

\section{Upper bound}\label{upperbound}

The upper bounds in Theorems~\ref{superweak} and \ref{prismsum} are
applications of the upper bounds on the volume of
right--angled hyperbolic polyhedra that were proved in \cite{volpoly}:

\begin{theorem}(\cite{volpoly}) 
\label{pi2generalupper} If $\PP$ is a $\pi/2$--equiangular
hyperbolic polyhedron, $N_{\infty}$ ideal
vertices, and $N_F\geq 1$ finite vertices, then 
$$ \vol(\PP) < \frac{N_{\infty}-1}{2}\cdot V_8 +
 \frac{5 N_F}{8}\cdot V_3 .$$ 
If all vertices of $\PP$ are ideal, then
$$\vol(\PP) < \frac{N_{\infty} - 4}{2}\cdot V_8.$$
\end{theorem}

To apply this theorem, we exhibit a volume--nondecreasing deformation from
any given non--obtuse hyperbolic polyhedron to a right--angled polyhedron.

Let $P$ be an abstract polyhedron.  Define $\MN_4(P)$ to
be the set of degree $4$ vertices, $\MN_2(P)$ to be the set of degree $3$
that are adjacent to three vertices of $\MN_4(P)$, and $\MN_3(P)$ to be the
degree $3$ vertices that are not contained in $\MN_2(P).$  For $i=2,3,4,$
let
$n_i(P) = |\MN_i(P)|$.  Define $\ME_{33}(P)$ to be the set of edges of $P$
with each endpoint in $\MN_3(P)$ and $\ME_{34}(P)$ to be the set of edges of
$P$ with one endpoint in $\MN_3(P)$ and the other endpoint in $\MN_4(P)$.
Define $E_{ij}(P) = |\ME_{ij}(P)|$.  Reference to $P$ will be suppressed when
the context is clear.  An observation that will prove useful is that any
edge not in $\ME_{33}$ is labeled by $\pi/2$.

The following is the main theorem of this section.

\begin{theorem}\label{genup}
Let $\PP$ be a non--obtuse hyperbolic polyhedron that realizes the labeled
abstract polyhedron $(P,\Theta)$.  Then
$$\vol(\PP) < \frac{n_4 + E_{33} -1}{2} \cdot V_8 + \frac{5(E_{34} + n_2)}{8} \cdot
V_3.$$
\end{theorem}

Let $(P,\Theta)$ be a labeled abstract polyhedron.  Define the \textbf{full
truncation} of $(P,\Theta)$ to be the right--angled abstract polyhedron 
$\widehat P$ obtained by replacing each vertex $v$ in $\MN_3$ by the triangle
formed by the midpoints of the edges entering $v$.  Each edge in $\ME_{33}$
is collapsed by this procedure.  See Figure~\ref{fulltrunc} for an example.
In the following lemma, we show that realizability of $(P,\Theta)$ implies
realizability of $(\wh P, \pi/2)$.

\begin{figure}
\labellist
\small\hair 2pt
\endlabellist
\begin{center}
\scalebox{1}{\includegraphics{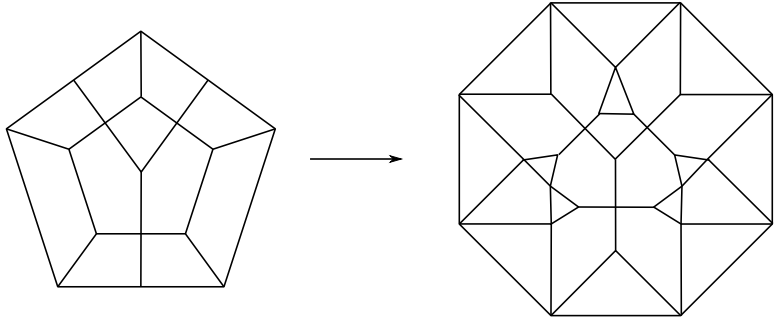}}
\end{center}
\caption{Full truncation}
\label{fulltrunc}
\end{figure}

\begin{lemma}\label{hypwhp}
If $(P,\Theta)$ is realizable as a hyperbolic polyhedron, then $(\widehat P,
\pi/2)$ is also realizable as a hyperbolic polyhedron.
\end{lemma}

\begin{proof}
The proof of this lemma amounts to showing that $(\wh P, \pi/2)$ satisfies
the conditions of Andreev's theorem restricted to right--angled polyhedra. 
For an abstract labeled polyhedron $(Q,\pi/2)$, Andreev's
theorem reduces to the following four conditions:  $P$ has at least $6$
faces, Each vertex has degree $3$ or degree $4$, $Q$ has no prismatic
$4$--circuits, and for any triple of faces $(F_i,F_j,F_k)$ such that $F_i\cap
F_j$ and $F_j\cap F_k$ are edges with distinct endpoints, $F_i \cap F_k =
\emptyset.$  

The number of faces of $\wh P$ is at least the
number of faces of $P$.  Hence, it is immediate that $\wh P$ has at least $6$
faces unless $P$ has the combinatorial type of a simplex or a triangular
prism in which case $\wh P$ has $8$ or $11$ faces, respectively.   The
fact that all vertices of $\wh P$ are degree $3$ or $4$ is immediate.

Suppose that $\wh P$ contains a triple of faces, $(F_i, F_j, F_k)$ such that
$e_{ij} = F_i \cap F_j$ and $e_{jk} = F_j \cap F_k$ are edges that have
distinct endpoints.  We show that $F_i \cap F_k = \emptyset$ as required by
Andreev's theorem.  

Assume for contradiction that $F_i \cap F_k \neq \emptyset$.  If $F_i \cap
F_k$ is an edge, $e_{ik}$,  then $e_{ij}$, $e_{jk}$ and $e_{ik}$ would form a
prismatic $3$--circuit.  The fact that a prismatic $3$--circuit may not pass
through a triangular face implies that these three edges correspond to edges
in $P$ that are not in $\ME_{33}$.  Hence in $P$, the corresponding edges
form a prismatic $3$--circuit with all three edges labeled by $\pi/2$.  This
is a contradiction to the assumption that $(P,\Theta)$ satisfies Andreev's
theorem.  If $F_i \cap F_k$ is an ideal vertex $v$, then there are two cases
to rule out.  The first case is that $F_i$ and $F_k$ are triangles that
arise as degenerations of vertices $v_1$ and $v_2$ of $P$.  Both $v_1$ and
$v_2$ would be vertices of the face corresponding to $F_j$ in $P$.  This
leads to a contradiction, however, for if $v_1$ and $v_2$ are adjacent in
$F_j$, they would be vertices of a bigon in $P$, and if $v_1$ and $v_2$ are
non--adjacent vertices of $F_j$, there would exist an edge of $P$ connecting
two non--adjacent vertices of $F_j$.  The second case is that $F_i$ and $F_k$
meet in an ideal vertex and do not arise as degenerations of vertices of $P$.
This leads to a contradiction because either the triple of faces in $P$
corresponding to $F_i,$ $F_j,$ and $F_k$ violate condition (7) of Andreev's
theorem or they form a spherical prismatic $3$--circuit.

Finally, any prismatic $4$--circuit in $\wh P$ cannot pass through any
triangular faces.  Hence, any edge traversed by a prismatic $4$-circuit in
$\wh P$ corresponds to an edge in $P$ that is not in $\ME_{33}.$  Andreev's
theorem precludes the existence of any such prismatic $4$-circuits in $P$,
which completes the proof.
\end{proof}

\begin{proof}[Proof of Theorem~\ref{genup}]
For $t \in [0,1)$, let  
$\Theta_t : \text{Edges}(P) \to (0,\pi/2]$ be a
labeling of $P$ defined by 
\begin{equation}
\Theta_t(e) = 
	\begin{cases}
	(1-t) \Theta(e) & \text{if }e \in E_{33}\\
	\Theta(e) = \pi/2 & \text{otherwise}.
	\end{cases}
\end{equation}
Let $\PP_t$ be
the hyperbolic realization of $(P^{\vee}, \Theta^{\vee}_t)$.  Recall from
Section~\ref{deformation} that $(P^{\vee}, \Theta^{\vee})$ is the labeled
abstract polyhedron where all vertices of $P$ around which the angle sum of
$\Theta$ is less than $\pi$ are truncated and $\Theta^{\vee}$ agrees with
$\Theta$ except along the edges of truncated faces where it assigns $\pi/2$.
For $t \in [0,1),$ it is clear that $(P,\Theta_t)$ satisfies the generalized
version of Andreev's theorem (Theorem~\ref{andbb}), so $\PP_t$ satisfies
Andreev's theorem for finite volume hyperbolic polyhedra.

By Schl\"{a}fli's formula and Milnor's continuity conjecture, the function
$\vol(\PP_t)$ is continuous and increasing in $t$.    There exists $t_0$ so
that for all $t > t_0$, each vertex in $\MN_3(P)$ is truncated in $\PP_t$.
For $t>t_0$, let $\QQ_t$ be the hyperbolic cone manifold obtained by doubling
$\PP_t$ along its faces.  By the proof of Thurston's generalized
hyperbolic Dehn filling theorem, $\QQ_t$ converges geometrically to $\wh \PP$
doubled along its faces as $t \to 1$ (See, for example, Appendix B of
\cite{bp}).  Therefore,
$\vol(\PP_t) \to \vol( \wh{\PP}).$
By Lemma~\ref{hypwhp}, $\wh \PP$ is hyperbolic, so applying Theorem~\ref{pi2generalupper} to $\wh \PP$ completes the proof.
\end{proof}

The following corollaries give the upper bounds in Theorems~\ref{superweak}
and \ref{prismsum}.

\begin{corollary}\label{weakupper}
Let $\PP$ be a non--obtuse hyperbolic polyhedron containing
$N_4$ degree $4$ vertices and $N_3$ degree $3$
vertices. Then
	$$\vol(\PP) < \frac{2N_4 + 3 N_3 -2}{4} \cdot V_8 + \frac{15N_3 + 20 N_4}{16}
	\cdot V_3.$$
\end{corollary}

\begin{proof}
Note that $2E_{33} \leq 3 n_3$, $2 E_{34} \leq 3 n_3 + 4 n_4$, $n_2
+ n_3 = N_3$, and $n_4 = N_4$.  The corollary then
follows from Theorem~\ref{genup} via a simple calculation:

\begin{align*}
\vol(\PP) & \leq \frac{n_4 + E_{33} -1}{2}\cdot V_8 + \frac{5(E_{34} +
n_2)}{8} \cdot V_3 \\ 
	& \leq \frac{2n_4 + 3 n_3 -2}{4} \cdot V_8 + \frac{5(3n_3 + 4 n_4 +
	2 n_2)}{16} \cdot V_3 \\
	& \leq \frac{2N_4 + 3 N_3 -2}{4} \cdot V_8 + \frac{5(3N_3 + 4 N_4)}{16}
	\cdot V_3.
\end{align*}
\end{proof}

\begin{corollary}
If $D_n$ is an $n$--prism, $n\geq 4$, then
$$\vol(D_n) < \frac{3n-4}{2}\cdot V_8.$$
\end{corollary}

\begin{proof}
All edges of an $n$-prism $D_n$ are in $\ME_{33}$, so $\widehat{D}_n$ is a
right--angled ideal polyhedron with $E_{33} = 3n$ vertices.  Apply the ideal
case of Theorem~\ref{pi2generalupper}.
\end{proof}

\section{Summary and an example}\label{summary}

In this section, we describe how to  estimate the volume of any
non--obtuse hyperbolic polyhedron $\PP$.   
In all cases, Theorem~\ref{genup} may be used to compute an upper bound for
the volume.

In the case where all angles are $\pi/3$ or less, the following theorem
follows from the discussion in Section~\ref{S:turndecomp} and a lower bound
on the volume of a $\pi/3$--equiangular polyhedron due to Rivin in a personal
communication.  A description of his argument is given in \cite{volpoly}.

\begin{theorem}\label{theorempi3}
If $\PP$ is a hyperbolic polyhedron with all dihedral angles less
than or equal to $\pi/3$, $N \geq 8$
vertices, and $M$ prismatic $3$--circuits, then
$$\vol(\PP) > (N+2M) \cdot \frac{3V_3}{8}.$$
\end{theorem}

If $\PP$ is an $n$--prism having no dihedral angles in the interval $(\pi/3,
\pi/2)$, then Theorem~\ref{prismsum} says that 
$$\vol(\PP) >  (n-3)\cdot \vol \left(C_1 \left(\frac \pi 3\right)\right),$$
where $\vol(C_1(\pi/3)) \approx .324423$.  If $\PP$ is an $n$--prism that
does have some dihedral angles in $(\pi/3, \pi/2)$, then the techniques of
Section~\ref{C:prisms} do not hold in their full generality, but may be
applied to any sub--cube of $\PP$ that has no dihedral angles in $(\pi/3,
\pi/2)$.

Otherwise, we first decompose along the collection of triangles and quadrilaterals
provided by Theorems~\ref{kmppolyorb} and \ref{bspolyorb} applied to
$\QQ_{\PP}^{\perp}.$
By Corollary~\ref{pi2deform}, each of the resulting atoroidal components may be deformed to right--angled
hyperbolic polyhedra with an additional ideal vertex for each quadrilateral
face that arose from the Bonahon--Siebenmann decomposition and an additional
finite vertex for each triangular face coming from the turnover
decomposition.  Theorem~\ref{pi2general} gives a lower bound for each of
these components.  For each of the prism--type components coming from the
Bonahon--Siebenmann decomposition, Theorem~\ref{mainthm} may be used to
obtain a lower bound.  In the case that a prism--type component contains
dihedral angles in the interval $(\pi/3, \pi/2)$, Theorem~\ref{mainthm} gives
a lower bound for any cube in the decomposition that contains no dihedral
angle in $(\pi/3, \pi/2)$.

\subsection{An example}\label{S:example}
We conclude by computing the estimates for an example.  The
initial polyhedron, $\PP$, is displayed on the left in
Figure~\ref{bigexample}.  The first step in computing the lower bound is to
find a maximal collection of disjoint prismatic $3$--circuits.  For this
example, there are just two.  They are the dashed curves in the diagram on
the right in Figure~\ref{bigexample}.

\begin{figure} 
\labellist
\small\hair 2pt

\pinlabel $6$ [b] at 85 106

\pinlabel $4$ [r] at 34 30
\pinlabel $2$ [b] at 84 24
\pinlabel $5$ [tr] at 144 16
\pinlabel $4$ [tl] at 29 16
\pinlabel $37$ [br] at 106 259
\pinlabel $2$ [l] at 96 198
\pinlabel $2$ [br] at 70 228
\pinlabel $7$ [l] at 186 272
\pinlabel $2$ [tl] at 130 264
\pinlabel $2$ [bl] at 144 278
\pinlabel $2$ [r] at 140 228
\pinlabel $2$ [r] at 140 214
\pinlabel $2$ [l] at 194 260
\pinlabel $2$ [b] at 83 2

\endlabellist
\begin{center}
\scalebox{.65}{\includegraphics{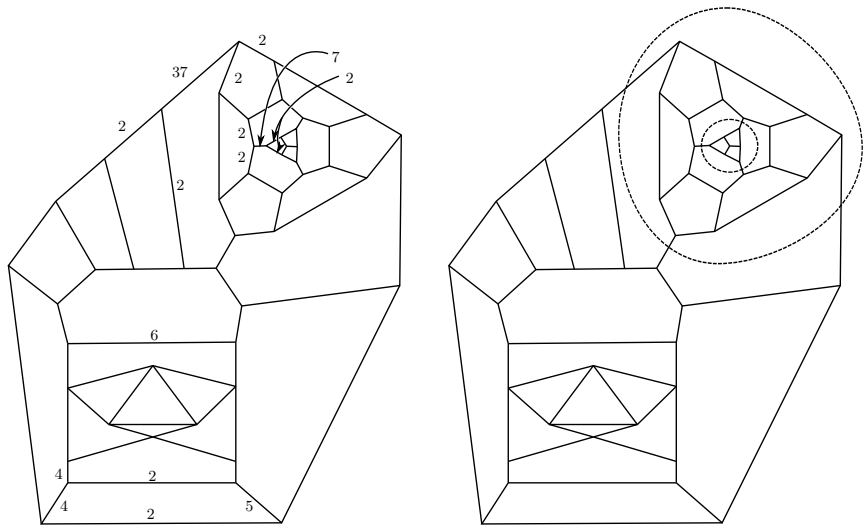}} 
\end{center}
\caption{Unlabeled edges that meet a degree four vertex are labeled $2$.
Unlabeled edges that meet a degree three vertex but not a degree four vertex
are labeled $3$} 
\label{bigexample} 
\end{figure}

The polyhedron $\PP$ is then decomposed along the corresponding turnovers, as
shown in the left--hand diagram in Figure~\ref{be3cutout}.  After capping off the
turnovers with orbifold balls, the small diagram is seen to be an order $4$
Coxeter prism, so has volume at least $\vol(C_1(\pi/3)) \approx .324423$.  
The other diagram that has been split off can be deformed to a compact
right--angled Coxeter polyhedron with $22$ vertices.  Therefore these two
components contribute at least
$\vol(C_1(\pi/3)) + \frac{7}{16}\cdot V_8$
to the volume of $\PP$.

\begin{figure} 
\begin{center}
\scalebox{.5}{\includegraphics{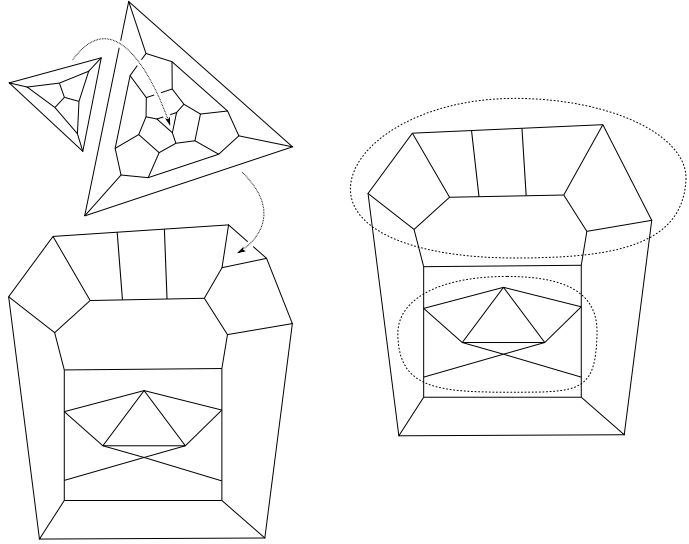}} 
\end{center}
\caption{The Coxeter cube is obtained by capping of the bounding triangle on the
upper--leftmost diagram. The diagram on the right shows the prismatic
$4$--circuits corresponding to the Bonahon--Siebenmann decomposition} 
\label{be3cutout} 
\end{figure}

The next step is to decompose along a subset of suborbifolds coming from the
Bonahon--Siebenmann decomposition into atoroidal and non--atoroidal components.
The result of part of this decomposition is seen in the left diagram in
Figure~\ref{be4cutout}. The
atoroidal component can be deformed to a right--angled polyhedron with $7$ ideal
vertices and $2$ finite vertices.  Therefore it contributes at least
$\frac{25}{16} \cdot V_8$
to the volume of $\PP$.
\begin{figure} 
\begin{center}
\scalebox{.5}{\includegraphics{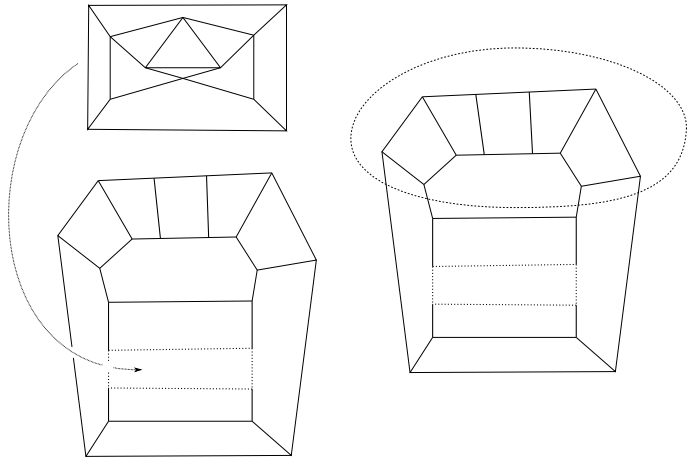}} 
\end{center}
\caption{These figures show the decomposition coming from the
Bonahon--Siebenmann splitting theorem} 
\label{be4cutout} 
\end{figure}

Finally, the remaining component, which is a reflection graph orbifold,
decomposes into two prism regions.  One of these prism regions contains only $6$
vertices of $\PP$.  Although we know that its volume is at least that of
$C_1(\mu)$ for some $\mu \in (0,\pi/2)$, we have not shown that $\mu$ is bounded
away from $\pi/2$, so the volume of $C_1(\mu)$ can be arbitrarily small.  The
other component has $12$ vertices so contributes at least
$$3 \cdot \vol(C_1(\pi/3))$$ to the volume of $\PP$.

\begin{figure} 
\begin{center}
\scalebox{.5}{\includegraphics{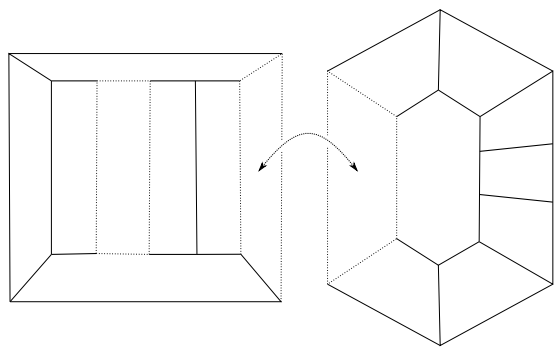}} 
\end{center}
\caption{This figure shows the decomposition of the graph orbifold region into
two prisms} 
\label{2prisms} 
\end{figure}

Adding these lower bounds together gives that the volume of $\PP$ is at least
$8.625$.

For the upper bound, we use Theorem~\ref{genup}.  For this example, $n_2=0$,
$n_4=6$, $E_{33} = 63$, and $E_{34}=6$.  This gives an upper bound of $128.377$.

\subsection{Acknowledgements} I wish to thank my thesis advisor, Ian Agol,
for many helpful conversations. I also wish to thank Dave Futer, Feng Luo,
Shawn Rafalski, and Louis Theran.

\bibliographystyle{plain}
\bibliography{./atkinson}
\end{document}